\documentclass[oneside,english,reqno]{amsart}

\usepackage[T1]{fontenc}
\usepackage[latin9]{inputenc}
\usepackage{babel}
\usepackage{amsthm}
\usepackage{amssymb}
\usepackage{graphicx}
\usepackage[unicode=true,pdfusetitle,
 bookmarks=true,bookmarksnumbered=false,bookmarksopen=false,
 breaklinks=false,pdfborder={0 0 1},backref=false,colorlinks=false]
 {hyperref}

\makeatletter

\providecommand{\tabularnewline}{\\}

\numberwithin{equation}{section}
\numberwithin{figure}{section}
\numberwithin{table}{section}
\theoremstyle{plain}
\newtheorem{thm}{\protect\theoremname}[section]
  \theoremstyle{definition}
  \newtheorem{defn}[thm]{\protect\definitionname}
  \theoremstyle{plain}
  \newtheorem{algorithm}[thm]{\protect\algorithmname}
  \theoremstyle{remark}
  \newtheorem{rem}[thm]{\protect\remarkname}
  \theoremstyle{plain}
  \newtheorem{prop}[thm]{\protect\propositionname}

\newcommand{\cone}{\mbox{\rm cone}}

\newcommand{\mspan}{\mbox{\rm span}}

\makeatother

  \providecommand{\algorithmname}{Algorithm}
  \providecommand{\definitionname}{Definition}
  \providecommand{\propositionname}{Proposition}
  \providecommand{\remarkname}{Remark}
\providecommand{\theoremname}{Theorem}

\begin{document}
\title[Set intersection problem algorithms]{Set intersection problems: Integrating projection and quadratic programming algorithms} 

\subjclass[2010]{90C25, 90C20, 47J25, 52A20}
\begin{abstract}
The Set Intersection Problem (SIP) is the problem of finding a point
in the intersection of convex sets. This problem is typically solved
by the method of alternating projections. To accelerate the convergence,
the idea of using Quadratic Programming (QP) to project a point onto
the intersection of halfspaces generated by the projection process
was discussed in earlier papers. This paper looks at how one can integrate
projection algorithms together with an active set QP algorithm. As
a byproduct of our analysis, we show how to accelerate an SIP algorithm
involving box constraints, and how to extend a version of the Algebraic
Reconstruction Technique (ART) while preserving finite convergence.
Lastly, the warmstart property of active set QP algorithms is a valuable
property for the problem of projecting onto the intersection of convex
sets.
\end{abstract}

\author{C.H. Jeffrey Pang}

\curraddr{Department of Mathematics\\ 
National University of Singapore\\ 
Block S17 08-11\\ 
10 Lower Kent Ridge Road\\ 
Singapore 119076 }

\email{matpchj@nus.edu.sg}

\date{\today{}}

\keywords{quadratic programming, projection algorithms, feasibility problems,
algebraic reconstruction technique.}

\maketitle
\tableofcontents{}

\section{Introduction}

For finitely many closed convex sets $K_{1},\dots,K_{r}$ in a Hilbert
space $X$, the \emph{Set Intersection Problem }(SIP) is stated as:
\begin{equation}
\mbox{(SIP):}\quad\mbox{Find }x\in K:=\bigcap_{i=1}^{r}K_{i}\mbox{, where }K\neq\emptyset.\label{eq:SIP}
\end{equation}
 One assumption on the sets $K_{i}$ is that projecting a point in
$X$ onto each $K_{i}$ is a relatively easy problem. 

A popular method of solving the SIP is the \emph{Method of Alternating
Projections }(MAP), where one iteratively projects a point through
the sets $K_{i}$ to find a point in $K$. Such methods fall under
the broad class of \emph{projection algorithms}. The rate of convergence
of the MAP is linear, and is known to be slow when certain angles
are small. Acceleration techniques include \cite{GPR67,GK89,BDHP03}.
For more on the background and recent developments of the MAP and
its variants, we refer the reader to \cite{BB96_survey,BR09,EsRa11},
as well as \cite[Chapter 9]{Deustch01}, \cite[Subsubsection 4.5.4]{BZ05},
and \cite{Deutsch01_survey,Deutsch95}. Superlinear convergence can
be achieved if the smoothness of the sets $K_{i}$ are exploited \cite{G-P98,G-P01}.
Another acceleration method is the Dos Santos method \cite{DosSantos87,DePierro81},
which is based on Cimmino's method for linear equations \cite{Cimmino38}.

Another problem related to the SIP is the \emph{Best Approximation
Problem }(BAP): Find the closest point to $x_{0}$ in $K$, that is,
\begin{eqnarray}
\mbox{(BAP):}\quad & \underset{x\in X}{\min} & \|x-x_{0}\|\label{eq:Proj-pblm}\\
 & \mbox{s.t. } & x\in K:=\bigcap_{i=1}^{r}K_{i}.\nonumber 
\end{eqnarray}

The projection of a point onto a closed convex set generates a supporting
halfspace. One can then project an iterate onto the intersection of
these halfspaces using quadratic programming (QP). This gives us the
\emph{supporting halfspace and quadratic programming} (SHQP) strategy,
which is illustrated in Figure \ref{fig:alt-proj-compare}. In our
earlier paper on this topic, we mentioned the Dykstra's algorithm
\cite{Dykstra83,BD86} and the related paper \cite{Han88}. An alternative
to Dykstra's algorithm is the approach of \cite{Haugazeau68}, which
has been tried in \cite{Combettes_blockiter} for example. See also
\cite{BauschkeCombettes11}.

In \cite{cut_Pang12}, we proved that for the SIP, can, provided that
the point of convergence satisfies a linear regularity property, converge
to a point in the intersection of the sets \emph{multiple-term superlinearly.
}We say that a sequence $\{x_{i}\}$ converges to $\bar{x}$ \emph{multiple-term
superlinearly }if 
\[
\lim_{i\to\infty}x_{i}=\bar{x}\mbox{, and }\lim_{i\to\infty}\frac{\|x_{i+p}-\bar{x}\|}{\|x_{i}-\bar{x}\|}=0\mbox{ for some finite }p>0.
\]

In a follow-up \cite{improved_SIP}, we prove that this convergence
can actually be multiple-term quadratic if the sets satisfy a second
order regularity at the point of convergence, and how one can take
fewer halfspaces and still achieve fast convergence. These results
generalize those of \cite{G-P98,G-P01}, and give an indication that
the SHQP strategy is a good idea to accelerate the convergence of
projection algorithms, even though one may have to solve smaller QPs
formed from fewer halfspaces than what can be proved to achieve the
fast convergence. Particular cases of the SHQP strategy had been
studied before. In \cite{Pierra84}, Pierra suggested an extrapolation
using this idea when one of the sets is an affine space. This idea
was further studied in \cite{BausCombKruk06} as the EAPM.

For the case when the sets $K_{i}$ are affine spaces, the early papers
\cite{Neumann50,Halperin62} show that the MAP converges strongly
to the solution of the BAP. We discuss more about how to handle an
SIP with affine constraints in Remark \ref{rem:aff-constraints}.

It appears that prevailing algorithms for the SIP (see for example
the algorithms in \cite{CensorChenCombettesDavidiHerman12,EsRa11})
do not exploit smoothness of the sets and fall back to a Newton-like
method and achieve superlinear convergence in the manner of \cite{G-P98,G-P01}
and the algorithms of this paper.

\begin{figure}[h]
\includegraphics[scale=0.4]{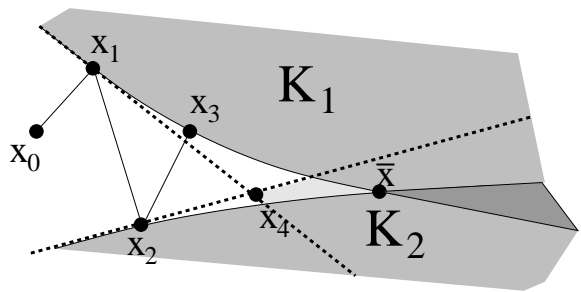}

\caption{\label{fig:alt-proj-compare}The method of alternating projections
on two convex sets $K_{1}$ and $K_{2}$ in $\mathbb{R}^{2}$ with
starting iterate $x_{0}$ arrives at $x_{3}$ in three iterations.
The point $x_{4}$ is the projection of $x_{1}$ onto the intersection
of halfspaces generated by projecting onto $K_{1}$ and $K_{2}$ earlier.
One can see that $d(x_{4},K_{1}\cap K_{2})<d(x_{3},K_{1}\cap K_{2})$,
illustrating the potential of the SHQP (supporting halfspace and quadratic
programming) strategy elaborated in \cite{cut_Pang12,improved_SIP}. }
\end{figure}

\subsection{Finite convergence of the ART}

The ART (Algebraic Reconstruction Technique) was originally designed
in \cite{GordonBenderHerman70ART} to find a solution to $Ax=b$ arising
in computerized tomography, where the rows of $A$ are formed during
the reconstruction process. This method is also known as the Kaczmarz
method \cite{Kaczmarz} in numerical linear algebra. The ART and Kaczmarz
method are updated by a projection method. 

A variant of the ART \cite{Herman1975}, referred to as ART3, finds
an $x\in\mathbb{R}^{n}$ that satisfies the linear inequality $L\leq Ax\leq U$.
If the set of points feasible in this system has nonempty interior
in $\mathbb{R}^{n}$, then ART3 finds such an $x$ in finitely many
iterations. The ART3 was designed to solve a reconstruction problem
from noisy x-rays.

\subsection{Algorithms for Quadratic Programs}

Algorithms for QPs have analogous counterparts to algorithms for linear
programs (LPs). Ellipsoid and Interior point methods for QPs and LPs
achieve polynomial time convergence. Another algorithm for LP is the
simplex method, which has exponential worst case behavior but is known
to perform well in practice. The simplex method has another advantage
over interior point methods in that it can make use of warmstart solutions
to solve the new problems created when constraints are added or deleted. 

The dual active set QP algorithm of \cite{Goldfarb_Idnani} can make
use of warmstart solutions to solve QPs when new linear inequality
constraints are added. We shall refer to their algorithm as the \emph{GI
algorithm}. This property of the GI algorithm makes it desirable to
solve the QPs arising from the BAP, since the projection process continually
adds new halfspaces bounding the intersection of the sets. Though
active set methods for QPs are not as extensively tested as that for
LPs, it is believed that, just like the simplex method for LPs, active
set methods can achieve good performance in practice, but have bad
worst case behavior.

\subsection{\label{sub:contributions}Contributions of this paper}

The theoretical analysis and possible implementation of the SHQP strategy
for the BAP and SIP have contrasting requirements. The multiple-term
fast convergence of the SHQP strategy for the SIP requires that a
huge number of recently produced halfspaces be used to form the QP.
On the other hand, practical concerns of solving a QP require that
the number of halfspaces used to form the QP cannot be too large. 

This paper looks at how one can integrate parts of the GI algorithm
into projection algorithms for the SIP and BAP. The GI algorithm is
chosen because it, being an active set algorithm, makes good use of
warmstart solutions when new inequality constraints are added (by
the projection process or by revisiting old constraints).  The GI
algorithm chooses to keep active halfspaces, which appears to be the
best choice of halfspaces to keep, and updates the QR factorization
of their normals for an effective solution of the QP.

To design our algorithms for the BAP and SIP, we isolate a step in
the algorithm in the dual active set QP algorithm of \cite{Goldfarb_Idnani},
calling it the inner GI step (Algorithm \ref{alg:Inner-GI}). The
additional work of the inner GI step is to update the KKT multipliers
and QR factorization of the normals of the active set. This additional
work is not too much more expensive than a projection if the number
of normals involved in the QR factorization is small (See Remark \ref{rem:cost_GI_step}),
and can be worth the effort when it accelerates convergence. For example,
multiple-term superlinear/ quadratic convergence is possible under
added conditions as proved in \cite{improved_SIP}, and such convergence
is also observed in our simple example in Section \ref{sec:simple-exa}.
In the case of the SIP, an inner GI step tries to move from one iterate
to the next by sliding along the halfspaces of some (possibly suboptimal)
choice of active constraints as far as possible. This is illustrated
in Figures \ref{fig:sliding} and \ref{fig:more_sliding}.

\begin{figure}[!h]
\includegraphics[scale=0.35]{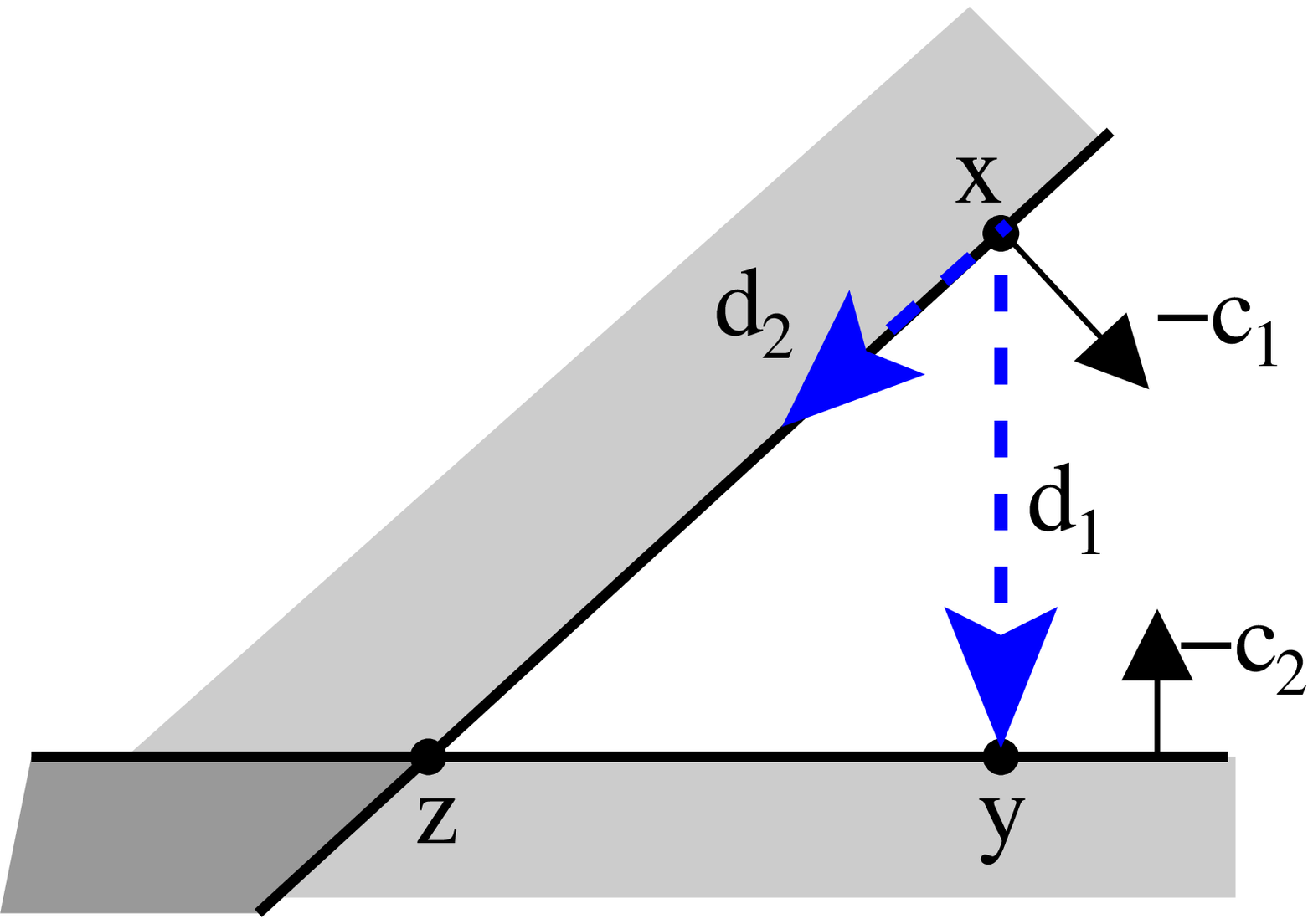}\includegraphics[scale=0.35]{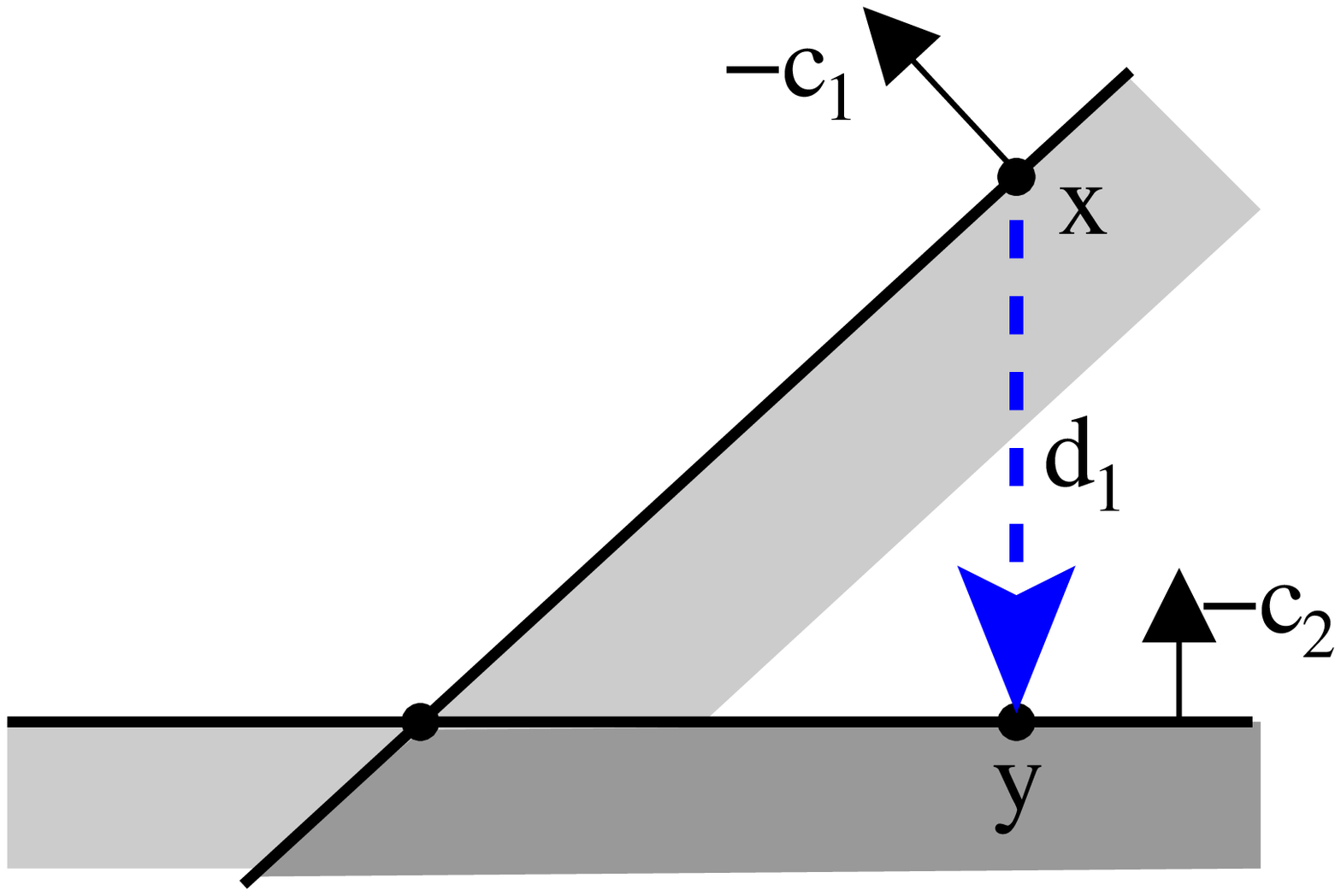}

\caption{\label{fig:sliding}In the diagram on the left, the projection of
$x$ onto some convex set produces a halfspace through $y$ with normal
$-c_{2}$. If we have prior information that $x$ lies on the boundary
of a halfspace with normal $-c_{1}$, then it is better to slide in
the direction $d_{2}$ along the boundary of the halfspace with normal
$-c_{1}$ and reach $z$. The point $z$ is also the projection of
$x$ onto the intersection of the two halfspaces, so this is one reason
why $z$ is a better iterate than $y$. Another reason why $z$ is
better than $y$ is that there are more halfspaces active at $z$
than at $y$, and one might be able to exploit this to improve convergence.
The diagram on the right shows that the constraint with normal $-c_{1}$
needs to be dropped in the next iteration. By integrating a QP algorithm
to a projection algorithm, one can take advantage of the situations
when the case on the left occurs.}

\end{figure}

The property of active set QP methods being able to effectively use
warmstart solutions to resolve a problem when new inequality constraints
are added is a useful feature for solving the BAP. See the description
of Algorithm \ref{alg:BAP-GI}. The GI algorithm is recommended over
other QP algorithms because even if the QPs were not solved fully,
the intermediate solutions of the GI algorithm improve as the number
of inner GI steps increases. See Remark \ref{rem:more_on_inner_GI}
and Figure \ref{fig:active-set}.

Noticing that the key to projecting onto a polyhedron using the inner
GI steps is to have a QR factorization of the normals of the active
constraints, we show an effective procedure for the SIP when one of
the sets is a box in Section \ref{sec:one-box}. Box constraints received
particular attention in \cite{CensorChenCombettesDavidiHerman12}.
We also point out how previous work \cite{BausCombKruk06,Pierra84},
\cite{HernandezRamosEscalanteRaydan2011} and \cite{Pang_subBAP}
can be applied to handle affine constraints.

In Section \ref{sec:modified-ART}, we extend the version of the ART
in \cite{Herman1975}, showing when one can use the SHQP ideas while
maintaining finite convergence. 

To summarize, this paper appears to be the first paper that studies
how one can integrate a QP algorithm and a projection algorithm for
solving the BAP and SIP. These techniques store the largest number
of the most useful halfspaces efficiently in order to obtain better
algorithms for the BAP and SIP. In particular, the new algorithms
can then fall back to a Newton-like method and achieve better than
linear convergence when the boundaries of the sets are sufficiently
smooth and satisfy a linear regularity property.

\subsection{Notation}

For a subset $S\subset\mathbb{R}^{n}$, we define $\mspan(S)$ to
be the \emph{span} of $S$ and $\cone(S)$ to be the \emph{conical
hull }of $S$ by 
\begin{eqnarray*}
\mspan(S) & := & \left\{ x:x=\sum\lambda_{i}s_{i},\,\lambda_{i}\in\mathbb{R},\, s_{i}\in S\right\} \\
\mbox{and }\cone(S) & := & \left\{ x:x=\sum\lambda_{i}s_{i},\,\lambda_{i}\geq0,\, s_{i}\in S\right\} .
\end{eqnarray*}
When $S$ is a matrix, $\mspan(S)$ and $\cone(S)$ are understood
to be the lineality space and conical hull generated by the columns
of $S$. Denote the projection of a point $x$ onto a closed convex
set $S$ to be $P_{S}(x)$. The QR factorization will always mean
the economy QR factorization. In other words, if $N\in\mathbb{R}^{n\times q}$,
then $N=QR$ for $Q\in\mathbb{R}^{n\times q}$ and $R\in\mathbb{R}^{q\times q}$.

\section{\label{sec:prelims}Preliminaries: Dual active set QP algorithm of
\cite{Goldfarb_Idnani}}

In this section, we recall the dual active set QP algorithm of \cite{Goldfarb_Idnani},
which will be the basis of the algorithms we discuss later. We also
refer to their algorithm as the GI algorithm. Their paper actually
treats the general case of a positive definite Hessian in the objective
function, but we shall just restrict to the case when the Hessian
is an identity matrix. 

\begin{figure}[!h]
\includegraphics[scale=0.3]{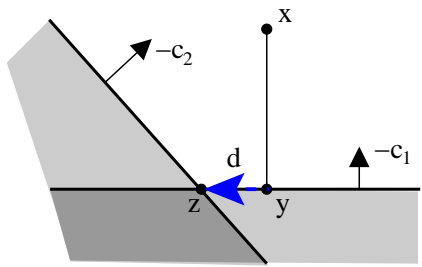}

\caption{This diagram illustrates the GI algorithm \cite{Goldfarb_Idnani}
to solve a quadratic program. The projection of x onto the halfspace
with normal $-c_{1}$ is $y$. When the halfspace that does not contain
$y$ is introduced, the GI algorithm finds a direction $d$ to move
from $y$ to $z$ (the projection of $x$ onto the intersection of
these two halfspaces) through what we call the inner GI step (Algorithm
\ref{alg:Inner-GI}). This direction $d$ is found by updating the
KKT multipliers. In Section 2, we recall the full details of the GI
algorithm, but omit derivations of its theoretical properties. For
more details, refer to \cite{Goldfarb_Idnani}, or even \cite{Idnani_thesis}.}

\end{figure}

For $x_{*}\in\mathbb{R}^{n}$, $C\in\mathbb{R}^{n\times m}$, and
$b\in\mathbb{R}^{m}$, consider the QP problem

\begin{eqnarray}
 & \min_{x\in\mathbb{R}^{n}} & v(x):=\frac{1}{2}\|x-x_{*}\|^{2}\label{eq:QP-pblm}\\
 & \mbox{s.t.} & C^{T}x\geq b.\nonumber 
\end{eqnarray}

We recall the GI algorithm and identify a step which we call the inner
GI step.
\begin{defn}
\label{def:s_tuple}(S-tuple) For the QP \eqref{eq:QP-pblm}, we say
that $(x,J,u,N,Q,R,q)$ is an \emph{s-tuple} if 
\begin{enumerate}
\item $J$ is an ordered subset of $\{1,\dots,m\}$ of size $q$ (and $q$
is an integer). 
\item $N\in\mathbb{R}^{q\times n}$ is such that the $i$th column of $N$
is the $J(i)$th column of $C$.
\item $(Q,R)$ is an economy QR factorization of $N$, where $Q\in\mathbb{R}^{n\times q}$
and $R\in\mathbb{R}^{q\times q}$. 
\item \cite{Goldfarb_Idnani} $(x,J)$ is an \emph{s-pair}, i.e., 

\begin{enumerate}
\item $x$ is the projection of $x_{*}$ onto $\{\tilde{x}:C_{J}^{T}\tilde{x}\geq b_{J}\}$,
and
\item $C_{J}^{T}x=b_{J}$. (i.e., $J$ is a subset of the constraints active
at $x$.)
\end{enumerate}
\item $u\in\mathbb{R}^{q}$ is the KKT multipliers certifying the optimality
of $x$, i.e.,

\begin{enumerate}
\item $u\geq0$, and
\item $x_{*}-x=-Nu$.
\end{enumerate}
\end{enumerate}
\end{defn}
The notion of the an s-pair in Definition \ref{def:s_tuple}(4) was
used in \cite{Goldfarb_Idnani}. We feel that it is useful to highlight
the other variables so that the computational costs of the dual active
set QP algorithm can be more easily calculated. 
\begin{algorithm}
\cite{Goldfarb_Idnani} \label{alg:GI-alg}(GI algorithm) This algorithm
finds the solution to the projection problem \eqref{eq:QP-pblm}.

\texttt{01} Set $J=\emptyset$, $x=x_{*}$, $q=0$, and set $u$,
$N$, $Q$ and $R$ to be empty vectors/ matrices

\texttt{02} $\quad$ so that $(x,J,u,N,Q,R,q)$ is an s-tuple.

\texttt{03} While $x$ is not feasible

\texttt{04} $\quad$ Find $p\in\{1,\dots,n\}$ such that $c_{p}^{T}x-b_{p}<0$

\texttt{05} $\quad$ Use the inner GI step to find an s-tuple $(x^{\prime},J^{\prime},u^{\prime},N^{\prime},Q^{\prime},R^{\prime},q^{\prime})$

\texttt{06} $\quad$ $\quad$ such that $v(x^{\prime})>v(x)$ or to
determine $\{\tilde{x}:C^{T}\tilde{x}\geq b\}=\emptyset$

\texttt{07} $\quad$ $\quad$ Let $(x,J,u,N,Q,R,q)\leftarrow(x^{\prime},J^{\prime},u^{\prime},N^{\prime},Q^{\prime},R^{\prime},q^{\prime})$

\texttt{08} end While.

\texttt{09} Return $(x,J,u)$.

\end{algorithm}
We go into more detail on the inner GI step.
\begin{algorithm}
\cite{Goldfarb_Idnani} \label{alg:Inner-GI}(Inner GI step) The \emph{Inner
GI step}, which is required in step 5 of Algorithm \ref{alg:GI-alg},
has the following inputs and outputs. We also list its full pseudocode
from \cite{Goldfarb_Idnani}.

\textbf{Inputs: }An s-tuple $(x,J,u,N,Q,R,q)$ and $p\notin J$ such
that $c_{p}^{T}x-b_{p}<0$.

\textbf{Outputs:} Either

(1) An s-tuple $(x^{\prime},J^{\prime},u^{\prime},N^{\prime},Q^{\prime},R^{\prime},q^{\prime})$
s.t. $J^{\prime}\subset J\cup\{p\}$ and $v(x^{\prime})>v(x)$, or

(2) A certificate of $\{\tilde{x}:C^{T}\tilde{x}\geq b\}$ being empty.

$\,$

Set $u^{+}=\left({u\atop 0}\right)$ to be a vector in $\mathbb{R}^{q+1}$.

\textbf{(a) Determine step direction}

Compute $z=[I-QQ^{T}]c_{p}$

If $q>0$, compute $r=R^{-1}Q^{T}c_{p}$

\textbf{(b) Compute step length}

\textbf{(i) Partial step length $t_{1}$ (maximum step in dual space
without violating dual feasibility)}

If $r\leq0$ or $q=0$, set $t_{1}\leftarrow\infty$

Otherwise, set 
\begin{equation}
t_{1}\leftarrow\min_{{r_{j}>0\atop j=1,\dots,q}}\left\{ \frac{u_{j}^{+}}{r_{j}}\right\} =\frac{u_{l}^{+}}{r_{l}}\label{eq:inner_GI_stuck}
\end{equation}

\textbf{(ii) Full step length $t_{2}$ (minimum step in primal space
such that the $p$th constraint becomes feasible.) }

If $|z|=0$, set $t_{2}\leftarrow\infty$.

Otherwise, set $t_{2}\leftarrow\frac{b_{p}-c_{p}^{T}x}{z^{T}c_{p}}$

\textbf{(iii) Step length $t$}

Set $t\leftarrow\min(t_{1},t_{2})$.

\textbf{(c) Determine new s-pair and take step}

\textbf{(i) No step in primal or dual space:}

If $t=\infty$, problem infeasible. Set $J^{\prime}\leftarrow J\cup\{p\}$.
We have $C_{J^{\prime}}\left({-r\atop 1}\right)=0$ and $\left({-r\atop 1}\right)\geq0$.

\textbf{(ii) Step in dual space:}

If $t_{2}=\infty$, then set $u^{+}\leftarrow u^{+}+t\left({-r\atop 1}\right)$,
and drop constraint $l$; i.e. set $J\leftarrow J\backslash\{J(l)\}$,
$q\leftarrow q-1$, update $Q$, $R$ and $N$ and go to step 2(a).

\textbf{(iii) Step in primal and dual space:}

Set $x\leftarrow x+tz$,

$u^{+}\leftarrow u^{+}+t\left({-r\atop 1}\right)$ .

If $t=t_{2}$, (full step), set $x^{\prime}\leftarrow x$, $u^{\prime}\leftarrow u^{+}$
and add constraint $p$; i.e. set $J^{\prime}\leftarrow J\cup\{p\}$,
$q^{\prime}\leftarrow q+1$, update $Q$, $R$ and $N$ to get $Q^{\prime}$,
$R^{\prime}$ and $N^{\prime}$, and inner GI step ends.

If $t=t_{1}$ (partial step), drop constraint $l$; i.e. set $J\leftarrow J\backslash\{J(l)\}$,
$q\leftarrow q-1$, update $Q$, $R$ and $N$ and go to step 2(a).

\end{algorithm}
\begin{rem}
(Certificate of $\{\tilde{x}:C^{T}\tilde{x}\geq b\}=\emptyset$) We
elaborate on why a certificate of infeasibility is obtained at step
(c)(i). At this step, we found an $r\leq0$ and a $J^{\prime}$ such
that $C_{J^{\prime}}\left({-r\atop 1}\right)=0$. At this iterate
$x$, we have $c_{j}^{T}x=b_{j}$ for all $j\in J^{\prime}$ and $c_{p}^{T}x<b_{p}$,
so 
\[
\left({-r\atop 1}\right)^{T}b_{J^{\prime}}>\left({-r\atop 1}\right)^{T}C_{J^{\prime}}^{T}x=0.
\]
For any $\tilde{x}$ such that $C_{J}^{T}\tilde{x}\geq b_{J}$, we
have $\left({-r\atop 1}\right)^{T}C_{J}^{T}\tilde{x}\geq\left({-r\atop 1}\right)^{T}b_{J}>0$.
But this is impossible since $C_{J^{\prime}}\left({-r\atop 1}\right)=0$
leads to a contradiction.
\end{rem}
Suppose $(x,J,u,N,Q,R,q)$ is an s-tuple. It is worth pointing out
again that the columns of $N$ form a subset of vectors which are
negative of the normals of halfspaces active at $x$. Furthermore,
$(Q,R)$ is the $QR$ factorization of $N$. The inner GI step maintains
the QR factorization of the normals active at $x$.
\begin{rem}
\label{rem:cost_GI_step}(Cost of inner GI step) A quick check on
the inner GI step shows that it involves the multiplication by $Q$,
$Q^{T}$ and $R^{-1}$ and the updating of the QR factorization of
$N$. (Since $R$ is upper triangular, the operation $r=R^{-1}Q^{T}c_{p}$
in step (a) is easy.) The most expensive operation in the inner GI
step is the updating of the QR factorization of $N$. This can be
performed with a series of at most $q$ Givens rotations when a single
column of $N$ is deleted, or by performing a Householder reflection
when a column is added to $N$. Such operations are standard in numerical
linear algebra, and are much cheaper than computing the QR factorization
of $N$ from scratch.
\end{rem}
We recall the finite convergence of using the GI algorithm to solve
\eqref{eq:SIP}.
\begin{thm}
\label{thm:fin-conv-GI}\cite{Goldfarb_Idnani} (Finite convergence
of GI algorithm) Consider Algorithm \ref{alg:GI-alg}, which solves
\eqref{eq:QP-pblm}.

(a) If $(x^{i},J^{i},u^{i},N^{i},Q^{i},R^{i},q^{i})$ and $(x^{i+1},J^{i+1},u^{i+1},N^{i+1},Q^{i+1},R^{i+1},q^{i+1})$
are successive s-tuples of Algorithm \ref{alg:GI-alg} to solve \eqref{eq:QP-pblm},
then $v(x^{i+1})>v(x^{i})$. Moreover, each $x^{i}$ is the projection
of $x_{*}$ onto the inequality constraints indexed by $J^{i}$, i.e.,
$x^{i}=P_{\{x:C_{J^{i}}^{T}x\geq b_{J^{i}}\}}(x_{*})$. 

(b) The GI algorithm converges to the optimal solution of \eqref{eq:QP-pblm}
in finitely many inner GI steps.
\end{thm}
Polynomial time algorithms for QPs include the ellipsoid method and
interior point methods. However, these methods are not as effective
in using warmstart solutions like active set methods. Though the performance
of QP algorithms are not as well tested as that for LPs (linear programs),
it is believed that active set methods behave like the simplex method:
Good in the average case, and possibly bad in the worst case. 
\begin{rem}
\label{rem:more_on_inner_GI}(More on inner GI steps) While Theorem
\ref{thm:fin-conv-GI}(b) says that the number of inner GI steps needed
to solve \eqref{eq:SIP} to optimality is finite, it may take too
many steps to solve the QP to optimality if the number of constraints
is large. Thankfully, the QPs in the SHQP strategy need not be solved
fully. Theorem \ref{thm:fin-conv-GI}(a) states that the next iterate
obtained by the inner GI step is a projection onto a polyhedron defined
by another set of inequalities that increases the distance $v(\cdot)$,
which makes the next iterate a better iterate than the last. Refer
to Figure \ref{fig:active-set}.

\begin{figure}[!h]
\includegraphics[scale=0.4]{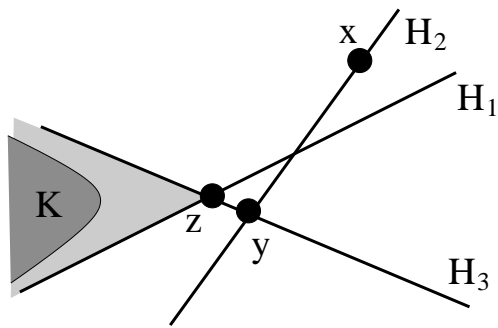}

\caption{\label{fig:active-set}Suppose at some point of an SHQP algorithm
that one arrives at the point $x$ that is on the boundary of some
halfspace $H_{2}$. The halfspace $H_{1}$ is a halfspace generated
by a previous projection, while $H_{3}$ is the newest halfspace generated.
An inner GI step from $x$ with $H_{3}$ as the violated constraint
leads to the point $y=P_{H_{2}\cap H_{3}}(x)$, and another inner
GI step from $y$ with $H_{1}$ as the violated constraint leads to
the point $z=P_{H_{1}\cap H_{3}}(x)=P_{H_{1}\cap H_{2}\cap H_{3}}(x)$.
Remark \ref{rem:more_on_inner_GI} says that even if one does not
land on the point $z=P_{H_{1}\cap H_{2}\cap H_{3}}(x)$ (i.e., the
exact solution of the QP subproblem), the point $y$ in the intermediate
calculation of the active set QP algorithm \cite{Goldfarb_Idnani}
is still a good point to continue calculations since $y$ is the projection
of $x$ onto the intersection of some of the halfspaces generated
so far. }

\end{figure}

\end{rem}

\section{Active set QP method for the BAP and SIP}

In this section, we describe our algorithms for the BAP and SIP using
the inner GI step. We can then adapt the GI algorithm for the BAP. 
\begin{algorithm}
\label{alg:BAP-GI}(Algorithm for the BAP) For a point $x_{0}$ and
closed convex sets $K_{l}$, $l=1,2,\dots,r$ in $\mathbb{R}^{n}$,
find the projection of $x_{0}$ onto $K:=\cap_{l=1}^{r}K_{l}$. 

\texttt{01} Set $J=\emptyset$, $x=x_{0}$, $q=0$, and set $u$,
$N$, $Q$ and $R$ to be empty vectors/ matrices

\texttt{02} $\quad$ so that $(x,J,u,N,Q,R,q)$ is an s-tuple.

\texttt{03} While $x\notin K$

\texttt{04} $\quad$ Find a halfspace $\{\tilde{x}:c_{p}^{T}\tilde{x}-b_{p}\leq0\}$
containing $K$ but not $x$.

\texttt{05} $\quad$ Find an s-tuple $(x^{\prime},J^{\prime},u^{\prime},N^{\prime},Q^{\prime},R^{\prime},q^{\prime})$

\texttt{06} $\quad$ $\quad$ such that $\|x^{\prime}-x_{0}\|>\|x-x_{0}\|$
or to determine $K=\emptyset$

\texttt{07} $\quad$ $\quad$ Let $(x,J,u,N,Q,R,q)\leftarrow(x^{\prime},J^{\prime},u^{\prime},N^{\prime},Q^{\prime},R^{\prime},q^{\prime})$

\texttt{08} end While.

\texttt{09} Return $x$.

\end{algorithm}
We now explain the differences of Algorithm \ref{alg:BAP-GI} from
Algorithm \ref{alg:GI-alg}.
\begin{rem}
\label{rem:line_4_BAP}(Line 4 of Algorithm \ref{alg:BAP-GI}) In
line 4 of Algorithm \ref{alg:BAP-GI}, a halfspace containing $\cap_{l=1}^{r}K_{l}$
but not $x$ is found by projecting $x$ onto a set $K_{l}$ such
that $x\notin K_{l}$. One can also revisit old constraints that were
dropped by earlier inner GI steps. Since $K$ is contained in the
intersection of the halfspaces obtained so far, a certificate of the
intersection of the halfspaces being empty also gives a certificate
of $K$ being empty in line 6.
\begin{rem}
\label{rem:line_5_BAP}(Line 5 of Algorithm \ref{alg:BAP-GI}) In
line 5, we no longer explicitly mention that one uses a single inner
GI step to find the next s-tuple. One can use more than one inner
GI step. As we have pointed out in Theorem \ref{thm:fin-conv-GI},
it takes finitely many inner GI steps to solve a QP with a set of
linear inequality constraints. As elaborated in Remark \ref{rem:more_on_inner_GI}It
may be worthwhile to solve a QP with a small number of constraints
to optimality, but once the number of constraints becomes too large,
it may be better off to find a new halfspace separating $K$ from
the current iterate $x$ to continue Algorithm \ref{alg:BAP-GI} instead.
\end{rem}
\end{rem}
We now propose our algorithm for the SIP using the inner GI step.
\begin{algorithm}
\label{alg:SIP-GI}(Algorithm for the SIP) For a point $x_{0}$ and
closed convex sets $K_{l}$, $l=1,2,\dots,r$ in $\mathbb{R}^{n}$,
find a point in $K:=\cap_{l=1}^{r}K_{l}$. The algorithm here is almost
exactly the same as that of Algorithm \ref{alg:BAP-GI}, except that
we change line 7 to:

\texttt{$\,$}

\texttt{05} $\quad$ Find an s-tuple $(x^{\prime},J^{\prime},u^{\prime},N^{\prime},Q^{\prime},R^{\prime},q^{\prime})$

\texttt{06} $\quad$ $\quad$ such that $\|x^{\prime}-x\|>0$ or to
determine $K=\emptyset$

\texttt{07} $\quad$ $\quad$ Let $(x,J,u,N,Q,R,q)\leftarrow(x^{\prime},J^{\prime},0,N^{\prime},Q^{\prime},R^{\prime},q^{\prime})$ 

$\,$
\end{algorithm}
Remarks \ref{rem:line_4_BAP} and \ref{rem:line_5_BAP} still apply
to Algorithm \ref{alg:SIP-GI}. In view of Definition \ref{def:s_tuple}(5)(b),
the point that we project from in Algorithm \ref{alg:SIP-GI} at the
end of line 7 is now changed to the current value of $x$ since $u$
is set to be the zero vector. This is reflected in Figure \ref{fig:sliding}
for two cases: When the inner GI step does not have to drop constraints,
and when the inner GI step has to drop a constraint. Since the point
we now project from is $x$, line 6 is changed accordingly.

Next, we explain why we need to use a different algorithm in line
5 of Algorithm \ref{alg:SIP-GI}. The inner GI step as written in
Algorithm \ref{alg:Inner-GI} requires $u>0$ in order to calculate
the ratio $t_{1}$ in \eqref{eq:inner_GI_stuck} properly, but in
Algorithm \ref{alg:SIP-GI}, the multipliers $u$ are often zero due
to line 7 of Algorithm \ref{alg:SIP-GI}. We write down an easy result
that leads to our design of the degenerate inner GI step (Algorithm
\ref{alg:degen_inner_GI}) to handle the case when $u=0$ at the start
of the inner GI step.
\begin{prop}
\label{prop:degen_QP}(A particular QP) Consider the QP 
\begin{eqnarray}
 & \min_{x\in\mathbb{R}^{n}} & \frac{1}{2}\|x-x_{*}\|^{2}\nonumber \\
 & \mbox{s.t. } & \tilde{C}^{T}x\geq\tilde{b}\label{eq:degen_QP_1}\\
 &  & \tilde{c}_{p}^{T}x\geq\hat{b},\nonumber 
\end{eqnarray}
where $\tilde{C}\in\mathbb{R}^{n\times q}$, $\tilde{b}\in\mathbb{R}^{q}$,
$\tilde{c}_{p}\in\mathbb{R}^{n}$ and $\hat{b}\in\mathbb{R}$ are
such that $\tilde{C}^{T}x_{*}=\tilde{b}$ and $\tilde{c}_{p}^{T}x_{*}<\hat{b}$.
Let $J$ be a subset of $\{1,\dots,q\}$, and let $N=\tilde{C}_{J}$
(i.e., the columns of $N$ are the columns of $\tilde{C}$ indexed
by $J$), with QR factorization $N=QR$. Let $z=(I-QQ^{T})\tilde{c}_{p}$,
$r=R^{-1}Q^{T}\tilde{c}_{p}$ and $t_{2}=\frac{\hat{b}-\tilde{c}_{p}^{T}x_{*}}{z^{T}\tilde{c}_{p}}$.
If $r\leq0$, then $x_{*}+t_{2}z$ is the optimal solution of the
problem 
\begin{eqnarray}
 & \min_{x\in\mathbb{R}^{n}} & \frac{1}{2}\|x-x_{*}\|^{2}\nonumber \\
 & \mbox{s.t. } & \tilde{C}_{J}^{T}x\geq\tilde{b}_{J}\label{eq:degen_QP_2}\\
 &  & \tilde{c}_{p}^{T}x\geq\hat{b},\nonumber 
\end{eqnarray}
and all constraints listed in \eqref{eq:degen_QP_2} are tight.\end{prop}
\begin{proof}
Before we continue with the proof, we remark that a result similar
to this is the foundation of why the inner GI step fits properly in
the GI algorithm. 

Let $x_{+}=x_{*}+t_{2}z$. Then 
\[
\tilde{C}_{J}^{T}x_{+}=\tilde{C}_{J}^{T}x_{*}+t_{2}\tilde{C}_{J}^{T}z=\tilde{b}_{J}+t_{2}N^{T}(I-QQ^{T})\tilde{c}_{p}.
\]
But $N^{T}(I-QQ^{T})=R^{T}Q^{T}(I-QQ^{T})=0$, so $\tilde{C}_{J}^{T}x_{+}=\tilde{b}_{J}$
as needed. Also, 
\[
\tilde{c}_{p}^{T}x_{+}=\tilde{c}_{p}^{T}x_{*}+\frac{\hat{b}-\tilde{c}_{p}^{T}x_{*}}{z^{T}n^{+}}\tilde{c}_{p}^{T}z=\hat{b}.
\]

Thus the constraints listed in \eqref{eq:degen_QP_2} are tight. Lastly,
\begin{eqnarray*}
x_{*}-x_{+} & = & t_{2}z\\
 & = & t_{2}(I-QQ^{T})\tilde{c}_{p}\\
 & = & t_{2}[\tilde{c}_{p}-QQ^{T}\tilde{c}_{p}]\\
 & = & t_{2}[\tilde{c}_{p}-QRr]\\
 & = & t_{2}[\tilde{c}_{p}-Nr]\\
 & = & t_{2}[\begin{array}{cc}
N & \tilde{c}_{p}\end{array}]\left(\begin{array}{c}
-r\\
1
\end{array}\right).
\end{eqnarray*}
All these show that the KKT conditions hold at $x_{+}$, and thus
$x_{+}$ is optimal.
\end{proof}
Such a $J$ found in Proposition \ref{prop:degen_QP} need not be
the active set that solves \eqref{eq:degen_QP_1}. To find this active
set, one would have to start from the solution in \eqref{eq:degen_QP_2}
and invoke inner GI steps using the constraints violated in $\tilde{C}^{T}x\geq\tilde{b}$
(that are certainly not in the index set $J$). Theorem \ref{thm:fin-conv-GI}
and Remark \ref{rem:more_on_inner_GI} still apply. 

We now write down the degenerate inner GI step in Algorithm \ref{alg:degen_inner_GI}
below, which plays the role of the inner GI step (Algorithm \ref{alg:Inner-GI})
when $u=0$. Algorithm \ref{alg:degen_inner_GI} is written so that
the similarities to the inner GI step are preserved.
\begin{algorithm}
\label{alg:degen_inner_GI}(Degenerate inner GI step) In step 5 of
Algorithm \ref{alg:SIP-GI}, one has to solve a QP of the form \eqref{eq:degen_QP_1}.
The inner GI step (Algorithm \ref{alg:Inner-GI}) cannot handle the
case when $u=0$, so we present this alternative algorithm. This algorithm
has the following inputs and outputs. We also list its full pseudocode.

\textbf{Inputs: }An s-tuple $(x,J,u,N,Q,R,q)$ and $p\notin J$ s.t.
$c_{p}^{T}x-b_{p}<0$ and $u=0$.

\textbf{Outputs:} Either

(1) An s-tuple $(x^{\prime},J^{\prime},u^{\prime},N^{\prime},Q^{\prime},R^{\prime},q^{\prime})$
s.t. $J^{\prime}\subset J\cup\{p\}$ and $v(x^{\prime})>v(x)$, or

(2) A certificate of $\{\tilde{x}:C^{T}\tilde{x}\geq b\}$ being empty.

$\,$

\textbf{(a) Determine step direction}

Loop

$\quad$Compute $z=[I-QQ^{T}]c_{p}$

$\quad$If $q>0$, compute $r=R^{-1}Q^{T}c_{p}$,

$\quad\quad$ otherwise $r=-1$.

$\quad$If $r\leq0$, break.

$\quad$Find index $l$ such that $r_{l}>0$.

$\quad$Update $J\leftarrow J\backslash\{J(l)\}$, $q\leftarrow q-1$.

$\quad$Remove $l$th column from $N$ and update $(Q,R)$.

end Loop

\textbf{(b) Compute full step length $t_{2}$ (minimum step in primal
space such that the $p$th constraint becomes feasible.) }

If $|z|=0$, then problem is infeasible. Set $J^{\prime}\leftarrow J\cup\{p\}$.
We have $C_{J^{\prime}}\left({-r\atop 1}\right)=0$ and $\left({-r\atop 1}\right)\geq0$.
End.

Otherwise, set $t_{2}\leftarrow\frac{b_{p}-c_{p}^{T}x}{z^{T}c_{p}}$

\textbf{(c) Determine new s-pair and take step in primal and dual
space}

Set $x^{\prime}\leftarrow x+t_{2}z$,

$u^{\prime}\leftarrow t_{2}\left({-r\atop 1}\right)$ .

Add constraint $p$; i.e. set $J^{\prime}\leftarrow J\cup\{p\}$,
$q^{\prime}\leftarrow q+1$, update $Q$, $R$ and $N$ to get $Q^{\prime}$,
$R^{\prime}$ and $N^{\prime}$, and degenerate inner GI step ends.

\end{algorithm}
We give a more detailed discussion of Algorithm \ref{alg:degen_inner_GI}
in Section \ref{sec:more_degen_inner_GI}.

We present an elementary result that shows the dimension reduction
property of the QR factorization that is applicable to both Algorithms
\ref{alg:BAP-GI} and \ref{alg:SIP-GI}.
\begin{prop}
\label{prop:reduc-1}(Problem size reduction) Let $x\in\mathbb{R}^{n}$,
$C\in\mathbb{R}^{n\times d}$ and $b\in\mathbb{R}^{d}$. Let $QR=C$
be the economy QR factorization of $C$. Suppose $z=P_{\{\tilde{x}^{\prime}:R^{T}\tilde{x}^{\prime}\geq b\}}(Q^{T}x)$.
Then $P_{\{\tilde{x}:C^{T}\tilde{x}\geq b\}}(x)=Qz+(I-QQ^{T})x$.\end{prop}
\begin{proof}
Let $y$ be any point be $\{\tilde{x}:C^{T}\tilde{x}\geq b\}$. Then
$C^{T}y\geq b$ implies that $R^{T}(Q^{T}y)\geq b$. So 
\begin{eqnarray*}
y-[Qz+(I-QQ^{T})x] & = & (I-QQ^{T})(y-x)+QQ^{T}y-Qz.\\
\mbox{and }x-[Qz+(I-QQ^{T})x] & = & QQ^{T}x-Qz.
\end{eqnarray*}
From $Q^{T}Q=I$, we have

\begin{equation}
[QQ^{T}y-Qz]^{T}[QQ^{T}x-Qz]=[Q^{T}y-z]^{T}[Q^{T}x-z].\label{eq:z-proj}
\end{equation}
Since $z=P_{\{\tilde{x}^{\prime}:R^{T}\tilde{x}^{\prime}\geq b\}}(Q^{T}x)$,
the term in \eqref{eq:z-proj} is less than equal to zero. Furthermore,
from $Q^{T}Q=I$, we have $(I-QQ^{T})Q=0$, so 
\begin{eqnarray*}
[(I-QQ^{T})(y-x)]^{T}[QQ^{T}x-Qz] & = & (y-x)^{T}(I-QQ^{T})Q[Q^{T}x-z]\\
 & = & 0.
\end{eqnarray*}
Combining the displayed formulas so far gives 
\[
\big[y-[Qz+(I-QQ^{T})x]\big]^{T}\big[x-[Qz+(I-QQ^{T})x]\big]\leq0.
\]
Since this holds for all $y\in\{\tilde{x}:C^{T}\tilde{x}\geq b\}$,
our conclusion follows.
\end{proof}
Proposition \ref{prop:reduc-1} shows that the QR factorization of
the constraints can be used to simplify the QP associated with projecting
a point $x$ onto the set of constraints generated so far. This is
a better strategy than following the descriptions of Algorithms \ref{alg:BAP-GI}
and \ref{alg:SIP-GI} to the letter. It appears that the kind of problems
where projection methods are the most effective (for example, those
in \cite{HermanChen08,CensorChenCombettesDavidiHerman12}) involve
a huge number of variables, so this dimensional reduction allows us
to find a better next iterate by solving a much smaller QP. 

We remark on the dropping and aggregation of constraints.
\begin{rem}
\label{rem:agg_BAP}(Aggregation of constraints for the BAP) In Algorithm
\ref{alg:BAP-GI} (BAP), one stores the s-tuple $(x,J,u,N,Q,R,q)$.
It may be necessary to store a smaller matrix $N$ (and correspondingly,
for all other variables except for $x$ and $q$) to reduce storage
requirements. We highlight Definition \ref{def:s_tuple}(5)(b) as
the condition requiring more attention. In the case where $N\in\mathbb{R}^{n\times2}$
and $u\in\mathbb{R}^{2}$, we can easily see how to perform the aggregation.
This same method of aggregation can be used to aggregate any two chosen
columns of $N$ when $N$ is of a much larger size. When $N$ is maintained
to be of size $n\times1$, Algorithm \ref{alg:BAP-GI} reduces to
Algorithm \ref{alg:Haugazeau} later. Just like in Remark \ref{rem:cost_GI_step},
we can make use of Givens rotations to update the QR factorizations
of the normals when columns are removed.
\begin{rem}
\label{rem:agg_SIP}(Dropping of constraints for the SIP) The aggregation
strategy in Remark \ref{rem:agg_BAP} can be done for Algorithm \ref{alg:SIP-GI}
as well. The s-tuple $(x,J,0,N,Q,R,q)$ has zeros for its $u$ component,
so the aggregation strategy can be much more flexible. Another strategy
to reduce storage requirements is to drop constraints. One can easily
see that when no normals of active constraints are stored, Algorithm
\ref{alg:SIP-GI} reduces to the method of alternating projections.

\end{rem}
\end{rem}
We recall another algorithm for the BAP. This method was tried in
\cite{Combettes_blockiter}, but it goes back to \cite{Haugazeau68}
as far as we are aware. See \cite{BauschkeCombettes11}.
\begin{algorithm}
\label{alg:Haugazeau}(An algorithm for the BAP) This algorithm solves
the BAP \eqref{eq:Proj-pblm}, i.e., the problem of projecting a point
$x^{0}$ onto the intersection $K=\cap_{l=1}^{r}K_{l}$ of convex
sets $K_{l}$. 

Set $i=0$

Loop

$\quad$Project $x^{i}$ onto a set $K_{l}$ and obtain a halfspace
$H_{i}$.

$\quad$Project $x^{0}$ onto the intersection of $H_{i}$ and the
halfspace 

$\quad\quad$with normal $x^{0}-x^{i}$ and $x^{i}$ onto its boundary
to get $x^{i+1}$.

$\quad$Set $i\leftarrow i+1$ 

until $x^{i}$ sufficiently close to $P_{K}(x^{0})$.
\end{algorithm}

\section{\label{sec:simple-exa}A simple example}

In this section, we make use of a simple example to show that Algorithms
\ref{alg:BAP-GI} (BAP) and \ref{alg:SIP-GI} (SIP) reduce to a Newton-like
method and achieve fast convergence. Methods that do not involve quadratic
programming (the method of alternating projections, Dykstra's algorithm)
or involve QPs that are too small (Algorithm \ref{alg:Haugazeau})
cannot reduce to a Newton-like method and achieve such fast convergence. 

Consider two circles with centers $(\pm2.9,0)$, radii $3$, and the
starting iterate of $(0,10)$. We apply Algorithms \ref{alg:BAP-GI}
(BAP) and \ref{alg:SIP-GI} (SIP). For both cases, the halfspaces
in step 4 are found by projecting onto the sets in a cyclic manner,
and we do not revisit old halfspaces. We only perform a single inner
GI step or degenerate inner GI step in step 5. It can be seen that
both sets of iterates converge to $\bar{x}=(0,\sqrt{0.59})$. Table
\ref{tab:circ-eg} shows the iterates of Algorithms \ref{alg:BAP-GI}
(BAP) and \ref{alg:SIP-GI} (SIP), with $x^{i}$ being the value of
$x$ in the $i$th iteration. The theory implies that multiple-term
superlinear convergence occurs, and this is observed in our experiment. 

We now compare our algorithms to the method of alternating projections.
With the same starting conditions, the method of alternating projections
takes 200 iterations to find a point distance $2.35\times10^{-13}$
from $\bar{x}$. The theory of alternating projections suggests that
the rate of linear convergence once close to $\bar{x}$ is 
\[
\cos\left(2\cos^{-1}\left(\frac{2.9}{3}\right)\right)=2\left(\frac{2.9}{3}\right)^{2}-1\approx0.8689.
\]
The corresponding values for measure 2 defined in Table \ref{tab:circ-eg}
would eventually be $\ln(0.8689)\approx-0.1405$, which was observed
in our numerical experiments. 

We now look at algorithms for the BAP. Algorithm \ref{alg:Haugazeau}
takes 90,000 iterations to find a point distance $7.51\times10^{-4}$
from $\bar{x}$. Dykstra's algorithm takes 2000 iterations to find
a point distance $3.97\times10^{-10}$ from $\bar{x}$, but we caution
that Dykstra's algorithm works differently in that the KKT conditions
are only approximately satisfied at each iterate.

\begin{table}[h]
\begin{tabular}{|l|l|l|l|l|l|l|}
\hline 
Iter & \multicolumn{2}{l|}{(a) $d(x^{i},\bar{x})$} & \multicolumn{2}{l|}{(b) Measure 1} & \multicolumn{2}{l|}{(c) Measure 2}\tabularnewline
\hline 
 & BAP & SIP & BAP & SIP & BAP & SIP\tabularnewline
\hline 
\hline 
0 & $9.23\times10^{0}$ & $9.23\times10^{0}$ &  &  &  & \tabularnewline
\hline 
1 & $2.95\times10^{0}$ & $2.95\times10^{0}$ & $-1.14$ & $-1.14$ & $-1.14$ & $-1.14$\tabularnewline
\hline 
2 & $1.48\times10^{0}$ & $7.98\times10^{-1}$ & $-0.917$ & $-1.22$ & $-0.694$ & $-1.31$\tabularnewline
\hline 
3 & $2.16\times10^{-1}$ & $1.70\times10^{-1}$ & $-1.25$ & $-1.33$ & $-1.92$ & $-1.55$\tabularnewline
\hline 
4 & $1.54\times10^{-1}$ & $7.57\times10^{-2}$ & $-1.02$ & $-1.20$ & $-0.338$ & $-0.809$\tabularnewline
\hline 
5 & $1.60\times10^{-2}$ & $8.04\times10^{-3}$ & $-1.27$ & $-1.41$ & $-2.26$ & $-2.24$\tabularnewline
\hline 
6 & $5.22\times10^{-3}$ & $1.38\times10^{-3}$ & $-1.25$ & $-1.47$ & $-1.12$ & $-1.76$\tabularnewline
\hline 
7 & $7.91\times10^{-5}$ & $1.79\times10^{-5}$ & $-1.67$ & $-1.88$ & $-4.19$ & $-4.34$\tabularnewline
\hline 
8 & $6.91\times10^{-6}$ & $4.84\times10^{-7}$ & $-1.76$ & $-2.10$ & $-2.44$ & $-3.61$\tabularnewline
\hline 
9 & $1.67\times10^{-9}$ & $8.28\times10^{-11}$ & $-2.49$ & $-2.83$ & $-8.33$ & $-8.67$\tabularnewline
\hline 
10 & $1.21\times10^{-11}$ & $5.93\times10^{-14}$ & $-2.74$ & $-3.27$ & $-4.93$ & $-7.24$\tabularnewline
\hline 
11 & $9.44\times10^{-16}$ & $7.86\times10^{-16}$ & $-3.35$ & $-3.36$ & $-9.46$ & $-4.32$\tabularnewline
\hline 
\end{tabular}

\caption{\label{tab:circ-eg}This table plots the values for the experiments
in Section \ref{sec:simple-exa}. To check the speed of convergence,
we calculate $\frac{[\ln(d(x^{i},\bar{x}))-\ln(d(x^{0},\bar{x}))]}{i}$
for Measure 1 and $\ln(d(x^{i},\bar{x}))-\ln(d(x^{i-1},\bar{x}))$
for Measure 2. In the case of linear convergence (of a geometric sequence),
these numbers will be constant. The measures decrease over iterations,
which is evidence for multiple-term superlinear convergence. }
\end{table}

\begin{rem}
(Insufficient halfspaces stored to get superlinear convergence) It
may appear from our example that collecting one halfspace from each
set in the SIP \eqref{eq:SIP} is sufficient to give multiple-term
superlinear convergence. But this is not true in general. An example
involving the intersection of a plane and a line in $\mathbb{R}^{3}$
where only linear and not superlinear convergence is obtained is given
in \cite[Example 4.7]{cut_Pang12}.
\end{rem}
An advantage of Algorithms \ref{alg:BAP-GI} and \ref{alg:SIP-GI}
is that they store and make use of previously generated halfspaces
as far as possible to accelerate convergence. Algorithm \ref{alg:Haugazeau}
and Dykstra's algorithm do not keep such information. One can perform
a single inner GI step in line 5 of Algorithm \ref{alg:BAP-GI} for
a low cost method of finding an iterate $x$. Alternatively, one can
revisit old constraints and perform as many inner GI steps as necessary
to solve the intermediate subproblems to optimality so that fewer
projections are carried out. Further discussions of Algorithm \ref{alg:SIP-GI}
are given in Section \ref{sec:more_degen_inner_GI}. 

Even if multiple-term superlinear convergence is not achieved, we
believe our algorithms would be better for most, if not all, problems.
When the problems are well structured, the number of halfspaces that
need to be stored to achieve multiple-term superlinear convergence
can be quite small.

\section{\label{sec:one-box}One box constraint}

The problem of interest in this section is 
\begin{eqnarray}
 & \min_{x\in\mathbb{R}^{n}} & \frac{1}{2}\|x-x_{*}\|^{2}\nonumber \\
 & \mbox{s.t. } & L_{i}\leq x_{i}\leq U_{i}\label{eq:box-degen-QP}\\
 &  & c_{p}^{T}x\geq\hat{b},\nonumber 
\end{eqnarray}
where $c_{p}\in\mathbb{R}^{n}$, $\hat{b}\in\mathbb{R}$, and $x_{*}$
satisfies $L_{i}\leq(x_{*})_{i}\leq U_{i}$ for all indices $i$,
but $c_{p}^{T}x<\hat{b}$. The constraints $L_{i}\leq x_{i}\leq U_{i}$
are also known as the \emph{box constraints}. Problems with box constraints
received particular attention in \cite{CensorChenCombettesDavidiHerman12}.
The $L_{i}$ and $U_{i}$ can be $\pm\infty$ rather than real numbers.
Such a problem arises when one of the sets involved in an SIP is a
box constraint, and the constraint $c_{p}^{T}x\geq\hat{b}$ is generated
by the projection of $x_{*}$ onto some set $K_{l}$ such that $x_{*}\notin K_{l}$.
The problem \eqref{eq:box-degen-QP} is a particular case of \eqref{eq:degen_QP_1}.
For \eqref{eq:degen_QP_1}, we can only find an active set $J$ that
solves \eqref{eq:degen_QP_1} partially in \eqref{eq:degen_QP_2}
through a degenerate inner GI step. But since the constraints corresponding
to the box constraints are orthogonal, we shall see in the rest of
this section that problem \eqref{eq:box-degen-QP} can be solved to
optimality quite effectively. This allows us to make maximum use of
the constraint $c_{p}^{T}x\geq\hat{b}$, especially when this constraint
is expensive to calculate.

\begin{figure}
\includegraphics[scale=0.5]{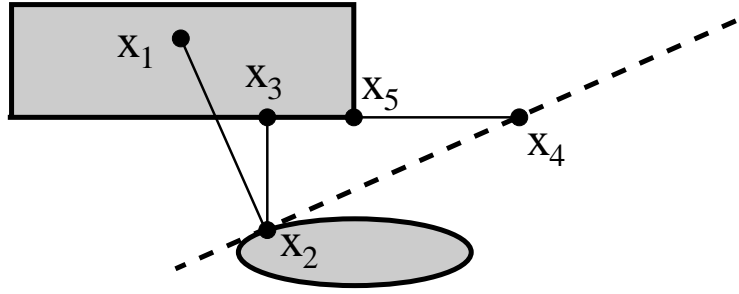}

\caption{This figure shows how one can make use of the orthogonal normals in
a box constraint to accelerate the SHQP procedure. The rectangle represents
a box constraint, while the ellipse represents a set that is not easy
to project onto. The first projection from $x_{1}$ to $x_{2}$ generates
a halfspace that must contain the intersection. Subsequently, one
tries to find a point in the box and the halfspace. The first three
iterations from $x_{1}$ to $x_{3}$ follow the standard alternating
projections. A tight contraint in the box is identified at $x_{3}$,
and one makes use of the procedure in Section \ref{sec:one-box} (motivated
by the idea in Figure \ref{fig:sliding}) to show that there is an
empty intersection.}
\end{figure}

To simplify notation, we let $B$ be the box 
\begin{equation}
B=\{x:L_{i}\leq x_{i}\leq U_{i}\mbox{ for all }i\}.\label{eq:def-box}
\end{equation}
 We now present our algorithm and result for \eqref{eq:box-degen-QP}.
\begin{algorithm}
\label{alg:box}(Algorithm for box constraints) This algorithm solves
\eqref{eq:box-degen-QP}. Let $B$ be as defined in \eqref{eq:def-box}.

\texttt{01} Set $j=0$ and $\tilde{x}^{0}=x_{*}$.

\texttt{02} While $c_{p}^{T}\tilde{x}^{j}<\hat{b}$

\texttt{03} $\quad$\textbf{(Determine step direction)} Let the direction
$d\in\mathbb{R}^{n}$ be defined to be 
\[
d_{i}=\begin{cases}
0 & \mbox{ if }\big[(c_{p})_{i}>0\mbox{ and }\tilde{x}_{i}^{j}=U_{i}\big]\mbox{ or }\big[(c_{p})_{i}<0\mbox{ and }\tilde{x}_{i}^{j}=L_{i}\big]\\
(c_{p})_{i} & \mbox{otherwise}.
\end{cases}
\]

\texttt{04} $\quad$\textbf{(Determine step length)} 

\texttt{05} $\quad$If $\|d\|=0$, then problem is infeasible. Stop.

\texttt{06} $\quad$Otherwise, let $t_{2}=\frac{\hat{b}-c_{p}^{T}\tilde{x}^{j}}{c_{p}^{T}d}$,
and let $y^{j}\leftarrow\tilde{x}^{j}+t_{2}d$. 

\texttt{07} $\quad$Let $\tilde{x}^{j+1}\leftarrow P_{B}(y^{j})$
and $j\leftarrow j+1$. 

\texttt{08} end while
\end{algorithm}
We now show that the iterates $\{\tilde{x}^{j}\}$ of Algorithm \ref{alg:box}
terminates to the solution of \eqref{eq:box-degen-QP}.
\begin{prop}
(Convergence of Algorithm \ref{alg:box}) Algorithm \ref{alg:box}
either converges to the solution of \eqref{eq:box-degen-QP} or declares
the infeasibility of its constraints in finitely many steps.\end{prop}
\begin{proof}
By multiplying components of $c_{p}$ by $-1$ and making the appropriate
adjustments to \eqref{eq:box-degen-QP} if necessary, we may assume
without loss of generality that $c_{p}\geq0$. Hence $d\geq0$.

We first prove that the algorithm cannot loop indefinitely. From the
formula $\tilde{x}^{j+1}=P_{B}(\tilde{x}^{j}+t_{2}d)$ and $d\geq0$,
we can infer that $c_{p}^{T}\tilde{x}^{j+1}>c_{p}^{T}\tilde{x}^{j}$.
It is also clear that $c_{p}^{T}\tilde{x}^{j+1}\leq c_{p}^{T}y^{j}=\hat{b}$.
By the design of $d$, an index active $i$ in $\tilde{x}^{j}$ (i.e.,
any index $i$ such that $\tilde{x}_{i}^{j}=U_{i}$) would result
in $d_{i}=0$, which in turn results in $i$ being an active index
in $\tilde{x}^{j+1}$. Furthermore, if $\tilde{x}^{j+1}\neq y^{j}$,
this means that there is an index active in $\tilde{x}^{j+1}$ that
is not active in $\tilde{x}^{j}$. There are only finitely many indices,
and so the algorithm cannot loop indefinitely.

To see that $\|d\|=0$ implies that the problem \eqref{eq:box-degen-QP}
is infeasible, note that $c_{p}$ would be a normal cone to $B$ at
$\tilde{x}^{j}$, signifying that the maximum value of $c_{p}^{T}x$
over all $x\in B$ equals $c_{p}^{T}\tilde{x}^{j}$, which is in turn
strictly less than $\hat{b}$. Thus the problem \eqref{eq:box-degen-QP}
is infeasible.

Lastly, we show that the final iterate of $\tilde{x}$ is the solution
of \eqref{eq:box-degen-QP}. Let the set $T_{j}$ be 
\[
T_{j}:=\mbox{conical hull of }[\{c_{p}\}\cup\{-e_{i}:i\mbox{ is an active index of }\tilde{x}^{j}\}].
\]
Note that by our earlier discussions, $T_{j+1}\supsetneq T_{j}$.
It is clear that $\tilde{x}^{j}-y^{j}\in T_{j}$ and $y^{j}-\tilde{x}^{j+1}\in T_{j+1}$.
Let $j^{*}$ be the final iteration of Algorithm \ref{alg:box}. 
We have $\tilde{x}^{0}-\tilde{x}^{j^{*}}\in T_{j^{*}}$, which gives
the the KKT conditions certifying the optimality of $\tilde{x}^{j^{*}}$
for \eqref{eq:box-degen-QP}. 
\end{proof}
In practice, one may stop at any of the $\tilde{x}^{j}$ or the $y^{j}$,
though it is sensible to perform all iterations of Algorithm \ref{alg:box}
and solve \eqref{eq:box-degen-QP} fully if the constraint $c_{p}^{T}x\geq\hat{b}$
was difficult to obtain.
\begin{rem}
(Comparison to Algorithm \ref{alg:SIP-GI}) Algorithm \ref{alg:box}
is actually Algorithm \ref{alg:SIP-GI} simplified for the case of
box constraints. Lines 3 to 6 correspond to the degenerate inner GI
step (Algorithm \ref{alg:degen_inner_GI}). Specifically, line 3 of
Algorithm \ref{alg:box} finds the optimal step direction, and corresponds
to step (a) of Algorithm \ref{alg:degen_inner_GI}, while lines 5
and 6 correspond to steps (b) and (c) of Algorithm \ref{alg:degen_inner_GI}.
Line 7 corresponds to a projection back onto the box. Besides Algorithm
\ref{alg:box} being much simpler than Algorithm \ref{alg:SIP-GI}
for the case of box constraints, the inner GI steps only require one
iteration to find the optimal step direction, and \eqref{eq:box-degen-QP}
can be solved fully after the appropriate number of steps.
\begin{rem}
\label{rem:aff-constraints}(Affine constraints) When some of the
sets in the SIP are affine constraints, the techniques of \cite{Pang_subBAP}
can be applied in a straightforward manner. Another approach for handling
affine spaces is \cite{HernandezRamosEscalanteRaydan2011}. Even though
the approach in \cite{HernandezRamosEscalanteRaydan2011} does not
appear to preserve Fej\'{e}r monotonicity, it uses a novel conjugate
gradient approach that gives some very good numerical results.  Idnani's
thesis \cite{Idnani_thesis} also elaborated on how to handle equality
constraints in the GI algorithm. 
\end{rem}
\end{rem}

\section{\label{sec:modified-ART}Finite convergence of a modified ART}

In this section, we look at finite convergence of a modified ART (Algebraic
Reconstruction Technique) algorithm. A version of the ART given in
\cite{Herman1975} solves the problem of finding a point in the intersection
of finitely many hyperslabs in $\mathbb{R}^{n}$, provided that the
intersection has nonempty interior. We look at how we can modify the
algorithm in \cite{Herman1975} so that one has the option of solving
QPs to accelerate performance, and still maintain finite convergence.

Consider the set 
\begin{equation}
S:=\{x\in\mathbb{R}^{n}:L\leq Ax\leq U\},\label{eq:rect-const}
\end{equation}
where $A\in\mathbb{R}^{m\times n}$ and each component of $L_{j}$
and $U_{j}$ lies in $\mathbb{R}\cup\{\pm\infty\}$ for $j\in\{1,\dots,m\}$.
We also let $S_{j}$ be the $j$th hyperslab defined by 
\[
S_{j}:=\{x\in\mathbb{R}^{n}:L_{j}\leq a_{j}^{T}x\leq U_{j}\},
\]
where $a_{j}^{T}$ is the $j$th row of $A$. The version of ART in
\cite{Herman1975}, referred in there as ART3, can be stated as follows:
\begin{algorithm}
\label{alg:ART_fin}\cite{Herman1975} (ART3) This algorithm finds
a point in $S$ defined in \eqref{eq:rect-const} in finitely many
iterations provided that $S$ has nonempty interior. Let $x^{0}\in\mathbb{R}^{n}$
be a starting iterate, and let $a_{j}^{T}$ be the $j$th row of $A$. 

\texttt{01} Set $i=0$, $j=1$, $k=m$

\texttt{02} Loop

\texttt{03} $\quad$if $j=k$, then $x\in S$ and STOP. Otherwise,
set $x^{i+1}$ by 
\begin{equation}
x^{i+1}=\begin{cases}
x^{i} & \mbox{if }a_{j}^{T}x^{i}\in[L_{j},U_{j}]\\
x^{i}+\frac{2(L_{j}-a_{j}^{T}x^{i})}{a_{j}^{T}a_{j}}a_{j} & \mbox{if }a_{j}^{T}x^{i}\in[L_{j}-\frac{1}{2}(U_{j}-L_{j}),L_{i}]\\
x^{i}+\frac{2(U_{j}-a_{j}^{T}x^{i})}{a_{j}^{T}a_{j}}a_{j} & \mbox{if }a_{j}^{T}x^{i}\in[U_{j},U_{j}+\frac{1}{2}(U_{j}-L_{j})]\\
x^{i}+\frac{\frac{1}{2}(L_{j}+U_{j})-a_{j}^{T}x^{i}}{a_{j}^{T}a_{j}}a_{j} & \mbox{otherwise.}
\end{cases}\label{eq:extrapolate}
\end{equation}

\texttt{04} $\quad$if $x^{i+1}\neq x^{i}$, then $k\leftarrow j$.

\texttt{05} $\quad$$i\leftarrow i+1$

\texttt{06} $\quad$$j\leftarrow\mbox{mod}(j+1,m)$. (i.e., $j\leftarrow j+1$
if $j<m$, and $j\leftarrow1$ if $j=m$.)

\texttt{07} end Loop

\end{algorithm}
We elaborate on the formula \eqref{eq:extrapolate}. The point $x^{i+1}$
is obtained from $x^{i}$ as follows: 
\begin{eqnarray*}
t & := & \max\big\{ t^{\prime}:P_{S_{j}}(x^{i})+t^{\prime}\big(P_{S_{j}}(x^{i})-x^{i}\big)\in S_{j}\big\}\\
\bar{t} & := & \min\{t/2,1\}\\
x^{i+1} & = & P_{S_{j}}(x^{i})+\bar{t}\big(P_{S_{j}}(x^{i})-x^{i}\big).
\end{eqnarray*}
The idea is to reflect $x^{i}$ about $P_{S_{j}}(x^{i})$ while staying
as far inside the hyperslab $S_{j}$ as possible.

The ART3 is proven to converge in finitely many iterations to a point
in $S$ as defined in \eqref{eq:rect-const}.
\begin{thm}
\cite{Herman1975} (Finite convergence of ART3) If $S$ in \eqref{eq:rect-const}
has nonempty interior, then Algorithm \ref{alg:ART_fin} converges
to a point in $S$ in finitely many iterations.
\end{thm}
We now extend Algorithm \ref{alg:ART_fin} so that one can perform
QP steps in an attempt to improve its convergence. To simplify notation,
define $\tilde{L}_{j}$ and $\tilde{U}_{j}$ as follows.

\begin{equation}
\tilde{L}_{j}:=L_{j}-\frac{1}{2}(U_{j}-L_{j}),\quad\tilde{U}_{j}:=U_{j}+\frac{1}{2}(U_{j}-L_{j}).\label{eq:tilde-L-U}
\end{equation}

In Algorithm \ref{alg:ext-ART} below, we shall keep a record of three
points $x_{\circ}^{i}$, $x_{\times}^{i}$ and $x_{+}^{i}$ at each
iteration $i$. The point $x_{\times}^{i}$ is the projection of $x_{\circ}^{i}$
onto a polyhedron defined by a few halfspaces that had been identified
earlier. The point $x_{+}^{i}$ is the extrapolation of this projection.
This is explained in Figure \ref{fig:ART}.
\begin{rem}
\label{rem:get-x-plus}(Finding $x_{+}^{i}$ given $x_{\circ}^{i}$
and $x_{\times}^{i}$) Refer to Figure \ref{fig:ART}. The point $x_{+}^{i}$
is found by extrapolating $x_{\times}^{i}$ along the direction $x_{\times}^{i}-x_{\circ}^{i}$
so that $x_{+}^{i}$ gets inside the interior of the intersection
of the hyperslabs active at $x_{\times}^{i}$ as far as possible.
More precisely, let $J^{i}$ be the set of indices of hyperslabs $S_{j}$
active at $x_{\times}^{i}$. Then calculate 
\begin{eqnarray*}
t & = & \max\{t^{\prime}:x_{\times}^{i}+t^{\prime}(x_{\times}^{i}-x_{\circ}^{i})\in S_{j}\mbox{ for all }j\in J^{i}\}\\
\bar{t} & = & \min\{t/2,1\}\\
x_{+}^{i} & = & x_{\times}^{i}+\bar{t}(x_{\times}^{i}-x_{\circ}^{i}).
\end{eqnarray*}

\begin{figure}[h]
\includegraphics[scale=0.4]{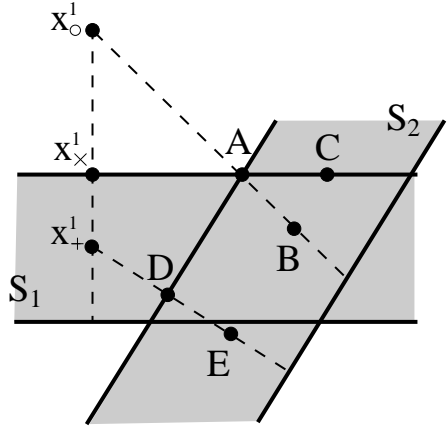}

\caption{\label{fig:ART}Consider the problem of finding a point in $S_{1}\cap S_{2}$,
where $S_{1}$ and $S_{2}$ are both hyperslabs. This diagram shows
how $x_{\circ}^{1}$, $x_{\times}^{1}$ and $x_{+}^{1}$ line up:
$x_{\times}^{1}=P_{S_{1}}(x_{\circ}^{1})$, and $x_{+}^{1}$ is of
the form $x_{\times}^{1}+t(x_{\times}^{1}-x_{\circ}^{1})$, where
$t$ is the largest number so that $x_{\times}^{1}+2t(x_{\times}^{1}-x_{\circ}^{1})\in S_{1}$.
\protect \\
When step ($P_{\circ}$) is performed, then $(x_{\circ}^{2},x_{\times}^{2},x_{+}^{2})=(x_{\circ}^{1},A,B)$.
\protect \\
When step ($P_{\times}$) is performed, then $(x_{\circ}^{2},x_{\times}^{2},x_{+}^{2})=(x_{\times}^{1},A,C)$.
\protect \\
When step ($P_{+}$) is performed, then $(x_{\circ}^{2},x_{\times}^{2},x_{+}^{2})=(x_{+}^{1},D,E)$.
\protect \\
If both ($P_{\circ}$) and ($P_{\times}$) are both possible, then
($P_{\times}$) is likely to be better because $x_{\circ}^{i+1}$,
$x_{\times}^{i+1}$ and $x_{+}^{i+1}$ are likely to have more constraints
active. (The exact situation depends on the inner GI steps.)\protect \\
If both ($P_{\circ}$) and ($P_{+}$) are possible, then ($P_{\circ}$)
is likely to be better because $x_{\times}^{i+1}$ and $x_{+}^{i+1}$
are likely to satisfy more constraints. }

\end{figure}
\end{rem}
\begin{algorithm}
\label{alg:ext-ART}(Extended ART) This algorithm finds a point in
$S$ defined in \eqref{eq:rect-const}. Let $x^{0}\in\mathbb{R}^{n}$
be a starting iterate, and let $a_{j}^{T}$ be the $j$th row of $A$.
Let $\tilde{L}_{j}$ and $\tilde{U}_{j}$ be defined as in \eqref{eq:tilde-L-U}.

Set $i=0$, $j=1$, $k=m$, $x_{\circ}^{0}=x^{0}$, $x_{\times}^{0}=x^{0}$
and $x_{+}^{0}=x^{0}$.

Loop

$\quad$Analyze cases:

$\quad$$\quad$\textbf{Case 1: }$a_{j}^{T}x_{+}\in[L_{j},U_{j}]$

$\quad$$\quad$$\quad$Do nothing: $x_{\circ}^{i+1}\leftarrow x_{\circ}^{i}$,
$x_{\times}^{i+1}\leftarrow x_{\times}^{i}$ and $x_{+}^{i+1}\leftarrow x_{+}^{i}$.

$\quad$$\quad$\textbf{Case 2:} $a_{j}^{T}x_{+}\in[\tilde{L}_{j},L_{j})\cup(U_{j},\tilde{U}_{j}]$
and $a_{j}^{T}x_{\times}\notin[L_{j},U_{j}]$

$\quad$$\quad$$\quad$Perform step ($P_{\circ}$) (preferred) or
($P_{+}$).

$\quad$$\quad$\textbf{Case 3:} $a_{j}^{T}x_{+}\in[\tilde{L}_{j},L_{j})\cup(U_{j},\tilde{U}_{j}]$
and $a_{j}^{T}x_{\times}\in[L_{j},U_{j}]$

$\quad$$\quad$$\quad$Perform step ($P_{+}$).

$\quad$$\quad$\textbf{Case 4:} $a_{j}^{T}x_{+}\notin[\tilde{L}_{j},\tilde{U}_{j}]$
and $a_{j}^{T}x_{\times}\notin[L_{j},U_{j}]$

$\quad$$\quad$$\quad$Perform step ($P_{\circ}$) or ($P_{\times}$)
(preferred) or ($P_{+}$).

$\quad$$\quad$\textbf{Case 5: $a_{j}^{T}x_{+}\notin[\tilde{L}_{j},\tilde{U}_{j}]$}
and $a_{j}^{T}x_{\times}\in[L_{j},U_{j}]$

$\quad$$\quad$$\quad$Perform step ($P_{\times}$) (preferred) or
($P_{+}$).

$\quad$end analyze cases.

$\quad$$i\leftarrow i+1$

$\quad$$j\leftarrow\mbox{mod}(j+1,m)$.

end Loop

$\quad$

\textbf{Step ($P_{\circ}$): Update $x_{\circ}^{i+1}$ from $x_{\circ}^{i}$.}

Feed the violated constraint (violated by $x_{\times}^{i}$) into
the inner GI step (Algorithm \ref{alg:Inner-GI}) to improve on the
problem of projecting from $x_{\circ}^{i}$. This is done by setting
$x_{\circ}^{i+1}=x_{\circ}^{i}$ and updating $x_{\times}^{i+1}$
by using inner GI steps from $x_{\times}^{i}$. Use Remark \ref{rem:get-x-plus}
to get $x_{+}^{i+1}$.

$\,$

\textbf{Step ($P_{\times}$): Update $x_{\circ}^{i+1}$ from $x_{\times}^{i}$.
We can build active sets with this step.}

Feed the violated constraint (violated by $x_{\times}^{i}$) into
the degenerate inner GI step (Algorithm \ref{alg:degen_inner_GI})
from the starting point $x_{\times}^{i}$ to get $x_{\circ}^{i+1}=x_{\times}^{i}$
and the new points $x_{\times}^{i+1}$ and $x_{+}^{i+1}$. Use Remark
\ref{rem:get-x-plus} to get $x_{+}^{i+1}$.

$\,$

\textbf{Step ($P_{+}$): Update $x_{\circ}^{i+1}$ from $x_{+}^{i}$.}

Feed the violated constraint (violated by $x_{+}^{i}$) into the degenerate
inner GI step (Algorithm \ref{alg:degen_inner_GI}) from the starting
point $x_{+}^{i}$ to get $x_{\circ}^{i+1}=x_{+}^{i}$ and the new
points $x_{\times}^{i+1}$ and $x_{+}^{i+1}$. Use Remark \ref{rem:get-x-plus}
to get $x_{+}^{i+1}$.

\end{algorithm}
The paper \cite{HermanChen08} pointed out that a cyclic order in
checking the constraints $a_{j}^{T}x\in[L_{j},U_{j}]$ is not necessarily
the most effective order. The strategy of keeping active sets through
the QP algorithm can be considered a way to improve on the cyclic
order in checking the constraints. 
\begin{rem}
(Preference of steps in Algorithm \ref{alg:ext-ART}) We list down
the reasons why we make certain choices about the steps performed
for each of the cases highlighted in Algorithm \ref{alg:ext-ART}.
\begin{enumerate}
\item (on ($P_{\times}$)) The reason why we do not perform step ($P_{\times}$)
for cases 2 and 3 (i.e., when $a_{j}^{T}x_{+}\in[\tilde{L}_{j},L_{j})\cup(U_{j},\tilde{U}_{j}]$)
is that performing ($P_{\times}$) disrupts the final part of the
proof of Theorem \ref{thm:fin-conv-ext-ART}. However, we still make
performing the step ($P_{\times}$) the preferred choice for cases
4 and 5. The paper \cite{Herman1975} cites \cite{Goffin71} in pointing
out that ``for rapid convergence, it is a good strategy to use projections
when we are far away from the hyperslab and reflections when we are
near it''. Another advantage of ($P_{\times}$) that we have seen
throughout this paper is that we can build active sets using the inner
GI steps that may accelerate the convergence of the problem at hand.
\item (on ($P_{\circ}$)) This step requires that $x_{\times}^{i}$ be outside
the slab $\{x:L_{j}\leq a_{j}^{T}x\leq U_{j}\}$ so that one can apply
the inner GI step with the appropriate face of the hyperslab being
used as a violated constraint. Therefore, ($P_{\circ}$) cannot be
performed in cases 3 and 5 when $a_{j}^{T}x_{\times}\in[L_{j},U_{j}]$.
In case 2, we marked that this step is preferred, but we recall Theorem
\ref{thm:fin-conv-GI} and Remark \ref{rem:more_on_inner_GI}. Furthermore,
if $x_{\times}$ is close to but not in the hyperslab, then it may
be better to perform ($P_{+}$) instead.
\end{enumerate}
\end{rem}
We now prove that the finite convergence of Algorithm \ref{alg:ext-ART}.
\begin{thm}
\label{thm:fin-conv-ext-ART}(Finite convergence of Algorithm \ref{alg:ext-ART})
Suppose that the feasible region $S$ defined in \eqref{eq:rect-const}
has nonempty interior. Then there is an $i$ such that $x_{+}^{i}\in S$
in Algorithm \ref{alg:ext-ART}.\end{thm}
\begin{proof}
Seeking a contradiction, suppose that $x_{+}^{i}\notin S$ for all
$i>0$. Rewrite the iterates $\{x_{\circ}^{i}\}$ as a subsequence
$\{\tilde{x}^{i}\}$ where for each $i$, either we have a projection
$\tilde{x}^{i+1}=P_{S_{i}}(\tilde{x}^{i})$ (when a step ($P_{\times}$)
is performed) or a partial reflection $\tilde{x}^{i+1}=P_{S_{i}}(\tilde{x}^{i})+\bar{t}(P_{S_{i}}(\tilde{x}^{i})-\tilde{x}^{i})$
for some $\bar{t}\in[0,1]$ (when a step ($P_{+}$) is performed).
Note that when a step ($P_{\circ}$) is performed at iteration $j$,
we have $x_{\circ}^{j+1}=x_{\circ}^{j}$, so we do not add a new term
for $\tilde{x}^{i}$. 

The number of times the step ($P_{\circ}$) is performed consecutively
cannot be infinite in view of Theorem \ref{thm:fin-conv-GI}. Therefore,
the sequence $\{\tilde{x}^{i}\}$ is infinite.

Since the sequence $\{\tilde{x}^{i}\}$ is formed by projections and
(partial) reflections about the convex sets containing $S$, it is
a standard result that the sequence $\{\tilde{x}^{i}\}$ is Fej\'{e}r
monotone with respect to $S$, i.e., 
\[
\|\tilde{x}^{i+1}-s\|\leq\|\tilde{x}^{i}-s\|\mbox{ for all }s\in S.
\]
Since $S$ was assumed to have nonempty interior, \cite[Theorem 4.5.10(iii)]{BZ05}
implies that the iterates $\{\tilde{x}^{i}\}$ converges to some point,
say $\bar{x}$. It is clear that this limit $\bar{x}$ cannot lie
outside any hyperslab $\{x:L_{j}\leq a_{j}^{T}x\leq U_{j}\}$, and
hence cannot lie outside $S$. If this limit $\bar{x}$ lies in the
interior of $S$, then we would have gotten finite convergence. Thus
this limit is on the boundary of $S$. 

There is a neighborhood $V$ of $\bar{x}$ such that for any $y\in V$
and any constraint $\{x:L_{j}\leq a_{j}^{T}x\leq U_{j}\}$ tight at
$\bar{x}$, we have $a_{j}^{T}y\in[\tilde{L}_{j},\tilde{U}_{j}]$,
where $\tilde{L_{j}}$ and $\tilde{U}_{j}$ are defined in \eqref{eq:tilde-L-U}.
Therefore once close enough to $\bar{x}$, step ($P_{\times}$) will
not be performed. This will mean that all the steps that are performed
are reflections. But these reflections would mean that $\{\|\tilde{x}^{i}-\bar{x}\|\}$
remains constant after some point, contradicting the assumption that
$\{\tilde{x}^{i}\}$ converges to $\bar{x}$. 
\end{proof}

\section{\label{sec:more_degen_inner_GI}More on the degenerate inner GI step}

In this section, we introduce an additional step (Algorithm \ref{alg:degen-GI-a-plus})
after step (a) of Algorithm \ref{alg:degen_inner_GI}, and show its
relationship with the inner GI step (Algorithm \ref{alg:Inner-GI})
and the degenerate inner GI step (Algorithm \ref{alg:degen_inner_GI})
in Theorem \ref{thm:a-plus-inn-GI}.  We also show how to reduce
the dimensions of the QP needed to be solved for Algorithm \ref{alg:SIP-GI}
(SIP). These observations can be built on to improve algorithms for
the SIP.

\begin{figure}[h]
\begin{tabular}{cc}
\begin{tabular}{c}
\includegraphics[scale=0.35]{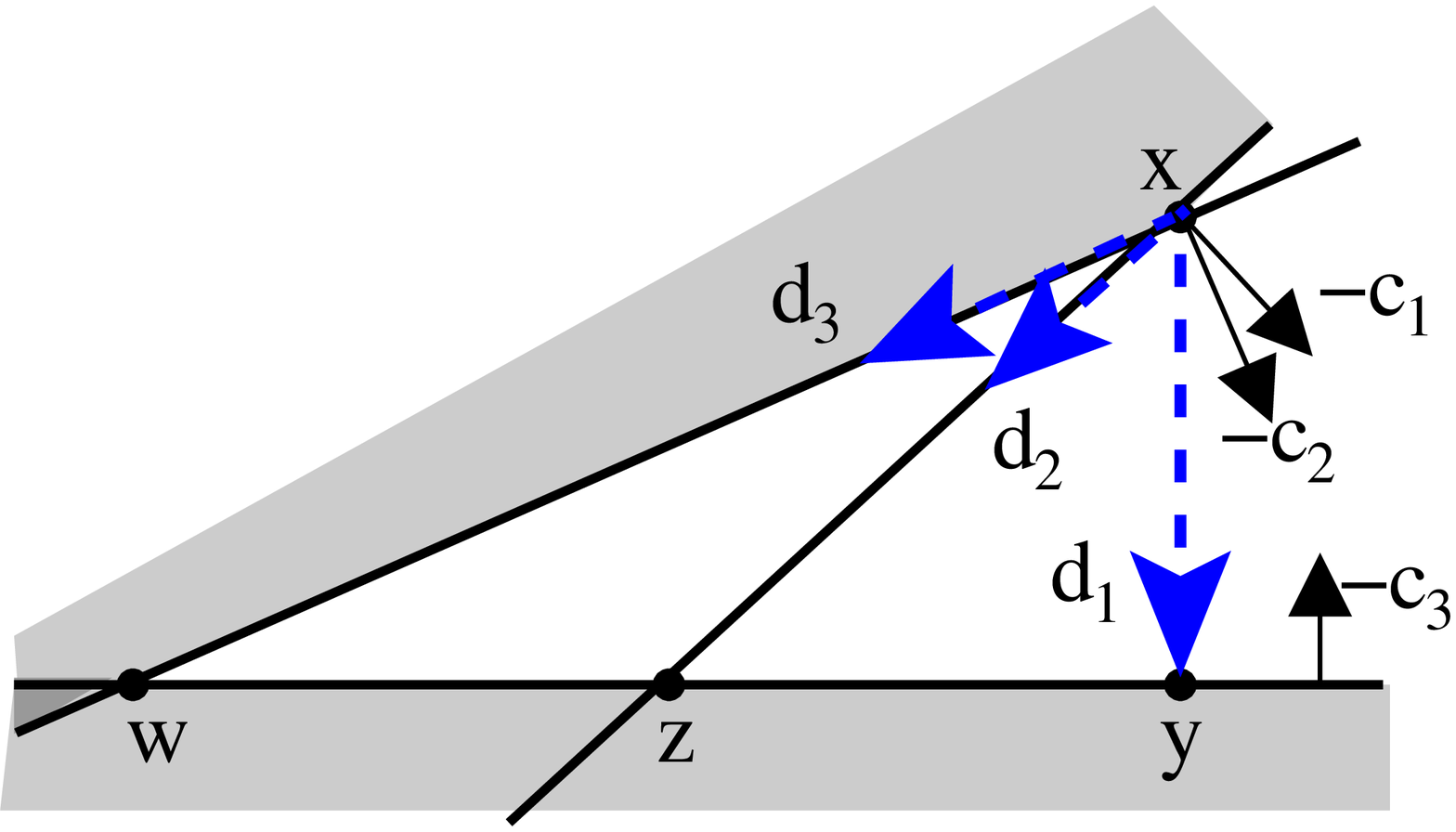}\tabularnewline
\end{tabular} & %
\begin{tabular}{c}
\includegraphics[scale=0.5]{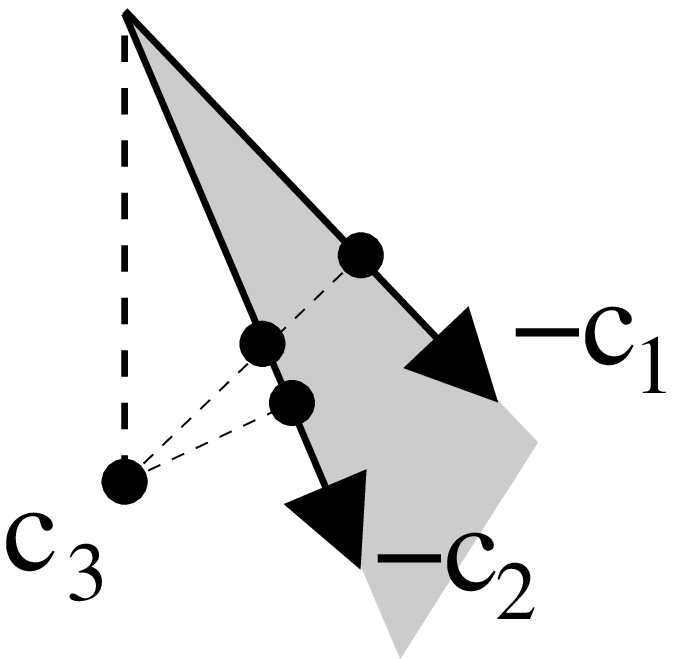}\tabularnewline
\end{tabular}\tabularnewline
\end{tabular}

\caption{\label{fig:more_sliding}In the diagram on the left, the degenerate
inner GI step (Algorithm \ref{alg:degen_inner_GI}) may not find the
correct set of active constraints to slide along in solving \eqref{eq:degen_QP_1}.
The degenerate inner GI step might find $z$ by sliding along the
constraints defined by the normal $c_{1}$, or find the point $y$
after dropping all constraints active at $x$ instead of finding the
optimal point $w$. One approach to find $w$ even if we didn't land
in $w$ in the first try is to introduce the violated constraints
into the inner GI step (Algorithm \ref{alg:Inner-GI}). Another approach
is to find a better direction to slide along by running as many steps
of Algorithm \ref{alg:degen-GI-a-plus} as needed between steps (a)
and (b) of Algorithm \ref{alg:degen_inner_GI}. The second approach
is equivalent to projecting $c_{3}$ onto the conical hull of $\{-c_{1},-c_{2}\}$,
as elaborated in the rest of Section \ref{sec:more_degen_inner_GI}.}
\end{figure}

\begin{prop}
\label{prop:z-of-degen-GI}(Step (a) of Algorithm \ref{alg:degen_inner_GI})
At the end of step (a) of Algorithm \ref{alg:degen_inner_GI},  $c_{p}-z=P_{\scriptsize\cone(-N)}(c_{p})$.\end{prop}
\begin{proof}
Note that $c_{p}-z=QQ^{T}c_{p}$ and $(Q,R)$ is the QR factorization
of $N$. So $c_{p}-z=P_{\scriptsize\mspan(-N)}(c_{p})$. Next, note
that by the time step (a) ends, $r\leq0$. Furthermore, 
\begin{eqnarray*}
r & = & R^{-1}Q^{T}c_{p}\\
\Rightarrow Nr & = & QQ^{T}c_{p}=c_{p}-z.
\end{eqnarray*}
In other words, $c_{p}-z$ is also $P_{\scriptsize\cone(-N)}(c_{p})$
as claimed.
\end{proof}
Proposition \ref{prop:z-of-degen-GI} states that the Algorithm \ref{alg:degen_inner_GI}
tries to find $z$ and the active set $J$ so that $c_{p}-z=P_{\scriptsize\cone(-N)}(c_{p})$.
(Readers familiar with the Moreau Theorem and the (negative) polar
cone will notice that an equivalent way to write this statement is
$z=P_{[\scriptsize\cone(-N)]^{-}}(c_{p})$, where $C^{-}$ refers
to the polar cone of $C$.) In trying to find this $z$ and $J$,
Algorithm \ref{alg:degen_inner_GI} may find a suboptimal active set
$J$ so that $c_{p}-z=P_{\scriptsize\cone(-N)}(c_{p})$ holds. 

Instead of running Algorithm \ref{alg:degen_inner_GI} and then running
Algorithm \ref{alg:Inner-GI} till \eqref{eq:degen_QP_1} is solved
(fully or partially; refer to Theorem \ref{thm:fin-conv-GI} and Remark
\ref{rem:more_on_inner_GI}), another approach is to try to solve\textrm{
$P_{\scriptsize\cone(-N)}(c_{p})$ }to optimality as far as possible
using a primal active set QP algorithm before the conclusion of Algorithm
\ref{alg:degen_inner_GI}. We now show how to do so.
\begin{algorithm}
\label{alg:degen-GI-a-plus}(After step (a) of Algorithm \ref{alg:degen_inner_GI})
Let $J_{0}$ and $N_{0}$ be the values of $J$ and $N$ at the start
of Algorithm \ref{alg:degen_inner_GI}, and let $J_{1}$ and $N_{1}$
be the values of $J$ and $N$ after step (a) of Algorithm \ref{alg:degen_inner_GI}
is performed. Note that $J_{1}\subset J_{0}$. This algorithm tries
to find the active set $J$ and the corresponding $N$ so that $P_{\scriptsize\cone(-N_{0})}(c_{p})=P_{\scriptsize\cone(-N)}(c_{p})$.
The direction $z=c_{p}-P_{\scriptsize\cone(-N)}(c_{p})$ to be used
at step (b) of Algorithm \ref{alg:degen_inner_GI} would then be the
best possible. 

We label this algorithm as step (a$^{+}$), which is to be performed
in Algorithm \ref{alg:degen_inner_GI} after step (a) and before step
(b).

\textbf{(a$^{+}$) Improve step direction}

\texttt{00} Set $J\leftarrow J_{1}$, and set $N$, $Q$, $R$ and
$r$ in a similar manner. Set $y\leftarrow Nr$.

\texttt{01} Loop

$\quad$\textbf{(Choose index to enter basis) }

\texttt{02} $\quad$Find an index $j\in J_{0}\backslash J$ such that
$[-c_{j}]^{T}[c_{p}-y]>0$

\texttt{03} $\quad\quad$If no such $j$ exists, then $y=P_{\scriptsize\cone(-N_{0})}(c_{p})$.
Exit loop.

\texttt{04} $\quad$Let $N_{+}=[\begin{array}{cc}
N & c_{j}\end{array}]$, 

\texttt{05} $\quad$Let $(Q_{+},R_{+})$ be the economy QR factorization
of $N_{+}$. 

\texttt{06} $\quad$Let $r^{+}=R_{+}^{-1}Q_{+}c_{p}$.

$\quad$\textbf{(Remove active indices till $j$ enters basis)}

\texttt{07} $\quad$While $r^{+}\not\leq0$

\texttt{08} $\quad\quad$Set 
\[
t_{3}:=\min_{{r_{i}^{+}>0\atop i\in\{1,\dots,q\}}}\left\{ \frac{-r_{i}}{r_{i}^{+}}\right\} =\frac{-r_{l}}{r_{l}^{+}}.
\]

\texttt{09} $\quad\quad$Let $r=r+t_{3}r^{+}$.

\texttt{10} $\quad\quad$Remove $J(l)$th element from $J$, and 

\texttt{11} $\quad\quad\quad$update $N$, $N_{+}$, $Q$, $Q_{+}$,
$R$, $R_{+}$, $r$ and $q$ accordingly.

\texttt{12} $\quad\quad$Let $r^{+}=R_{+}^{-1}Q_{+}c_{p}$

\texttt{13} $\quad$end While

\texttt{14} $\quad$Let $N\leftarrow N_{+}$, $r\leftarrow r^{+}$,
$Q\leftarrow Q_{+}$, $R\leftarrow R_{+}$, $q\leftarrow q+1$, $y\leftarrow Nr$.

\texttt{15} $\quad$If $y$ is deemed to be satisfactory, then break.

\texttt{16} end loop\end{algorithm}
\begin{rem}
(On Algorithm \ref{alg:degen-GI-a-plus})  Algorithm \ref{alg:degen-GI-a-plus}
describes a primal active set QP strategy for solving \eqref{eq:degen_QP_1},
where we try to find $y\in\cone(-N_{0})$ that minimizes the distance
from $c_{p}$ to $\cone(-N_{0})$. The standard theory of the active
set method implies that the distance $\|c_{p}-y\|$ is strictly decreasing
and also that there are only finitely many combinations of active
indices. Hence Algorithm \ref{alg:degen-GI-a-plus} must terminate
at the optimum $y=P_{\scriptsize\cone(-N_{0})}(c_{p})$ in finitely
many iterations. But the analogue of Theorem \ref{thm:fin-conv-GI}
and Remark \ref{rem:more_on_inner_GI} also apply for the primal active
set QP algorithm: It may be excessive to perform all the required
iterations to solve \eqref{eq:degen_QP_1} to optimality, but any
intermediate solutions can still give good progress towards the original
SIP.

A byproduct of this perspective for solving \eqref{eq:degen_QP_1}
is that the index $j$ that maximizes $[-c_{j}]^{T}[c_{p}-y]$ seems
be a good index to add in line 2 to accelerating convergence. A divide
and conquer approach can be used to find an approximation to $y=P_{\scriptsize\cone(-N_{0})}(c_{p})$
when the number of columns of $N_{0}$ is large. More details on the
primal active set QP method can be found in \cite{NW06} for example.
\end{rem}
We now prove a connection between Algorithm \ref{alg:degen-GI-a-plus}
and the inner GI step.
\begin{thm}
\label{thm:a-plus-inn-GI}(Step (a$^{+}$) and inner GI step) Consider
the first sequence of steps below
\begin{itemize}
\item Step (a) of Algorithm \ref{alg:degen_inner_GI} (degenerate inner
GI step), followed by an iteration of lines 2-15 of Algorithm \ref{alg:degen-GI-a-plus}
(step (a$^{+}$)), followed by the remainder of Algorithm \ref{alg:degen_inner_GI},
\end{itemize}
and the second sequence of steps below
\begin{itemize}
\item Algorithm \ref{alg:degen_inner_GI} (degenerate inner GI step) followed
by an inner GI step (Algorithm \ref{alg:Inner-GI}).
\end{itemize}
An index $j$ gives a violated constraint to perform step (a$^{+}$)
in the first sequence of steps if and only if this index gives a violated
constraint to perform an inner GI step in the second sequence of steps.
Furthermore, the two sequences of steps are equivalent provided that
the same violated constraint is used in both sequences of steps.\end{thm}
\begin{proof}
We mark the variables $N_{0}$, $Q_{0}$, $R_{0}$ and $J_{0}$ be
the corresponding parameters before step (a) of Algorithm \ref{alg:degen_inner_GI}
in both sequences of steps. Let the variables $N_{1}$, $Q_{1}$,
$R_{1}$, $J_{1}$ be the corresponding parameters after step (a)
and before the two sequences of steps diverges. 

By the end of step (a) of Algorithm \ref{alg:degen_inner_GI}, we
have found a direction $z=c_{p}-P_{\scriptsize\cone(-N_{1})}(c_{p})$. 

We look at a single iteration in the inner loop in step (a$^{+}$).
Step (a$^{+}$) finds a $j$ so that $c_{j}^{T}(c_{p}-y)>0$. 

Next, we consider the second sequence of steps. By the end of Algorithm
\ref{alg:degen_inner_GI}, we have found an $x$ such that $x^{+}=x+tz$,
where $t=\frac{b-c_{p}^{T}x}{c_{p}^{T}z}>0$. Making use of the connection
$c_{p}-y=z$, we have $-c_{j}^{T}(c_{p}-y)>0$ if and only if $c_{j}^{T}z<0$.
The condition $-c_{j}^{T}(c_{p}-y)>0$ is thus equivalent to 
\[
c_{j}^{T}x^{+}=c_{j}^{T}x+tc_{j}^{T}z=b_{j}+tc_{j}^{T}z<b_{j}.
\]
In other words, the first statement of our result holds. 

To prove the second statement, we look at the first sequence of steps
again. The inner loop of step (a$^{+}$) contains yet another inner
loop in lines 7-13. It suffices to compare this inner loop with the
inner loop within the inner GI step. 

In the second sequence of steps, consider the variables at the end
of performing Algorithm \ref{alg:degen_inner_GI}. From $x^{+}=x+tz$
in Algorithm \ref{alg:degen_inner_GI}, we have $z=\frac{1}{t}(x^{+}-x)$.
From the definition of an s-tuple in Definition \ref{def:s_tuple}(5)(b),
we have $z=-Nu$, where $u=\left({-r\atop 1}\right)\geq0$. The equivalences
of the two sequences of steps can be observed by noting the similarities
between the variables $r$, $r^{+}$ and $t_{3}$ in step (a$^{+}$)
and the variables $-u$ (the first $q$ components only), $r$ and
$t_{1}$ in the inner loop of the inner GI step.
\end{proof}
The elementary result below, which is quite similar to Proposition
\ref{prop:reduc-1}, makes a final point about reducing the problem
size to speed up the algorithm.
\begin{prop}
(Problem size reduction) Let $Q_{0}R_{0}$ be a QR factorization of
$N_{0}$. The columns of $N_{0}$ are indexed by $J_{0}$. Let $\tilde{y}=P_{\scriptsize\cone(-R_{0})}(Q_{0}^{T}c_{p})$.
Then 
\begin{enumerate}
\item There is an index set $J$ and a vector $r<0$ such that $\tilde{y}=(R_{0})_{J}r$,
where $(R_{0})_{J}$ is the matrix formed by the columns of $R_{0}$
indexed by $J$. 
\item $Q_{0}\tilde{y}=P_{\scriptsize\cone(-N_{0})}(c_{p})$.
\end{enumerate}
Lastly, if the QR factorization of $(R_{0})_{J}$ is obtained in the
process of finding $\tilde{y}$, say $(R_{0})_{J}=\tilde{Q}\tilde{R}$,
then $[Q_{0}\tilde{Q}]\tilde{R}$ would be the QR factorization of
$Q_{0}(R_{0})_{J}$.\end{prop}
\begin{proof}
Since $\tilde{y}=P_{\scriptsize\cone(-R_{0})}(Q_{0}^{T}c_{p})$, we
have $\tilde{y}\in\cone(-R_{0})$, so there is an index set $J\subset J_{0}$
and a vector $r<0$ such that $\tilde{y}=(R_{0})_{J}r$, giving us
the first conclusion. 

For the second conclusion, let $y=Q_{0}\tilde{y}$. We want to check
that $y=P_{\scriptsize\cone(-N_{0})}(c_{p})$. Any element in $\cone(-N_{0})$
can be written as $Q_{0}R_{0}\tilde{v}$ for some $\tilde{v}\leq0$.
Then 
\begin{eqnarray*}
[c_{p}-y]^{T}[Q_{0}R_{0}\tilde{v}-y] & = & [c_{p}-Q_{0}Q_{0}^{T}c_{p}]^{T}Q_{0}[R_{0}\tilde{v}-\tilde{y}]\\
 &  & +[Q_{0}Q_{0}^{T}c_{p}-y]^{T}Q_{0}[R_{0}\tilde{v}-\tilde{y}].
\end{eqnarray*}
The expressions above can be simplified 
\begin{eqnarray*}
[c_{p}-Q_{0}Q_{0}^{T}c_{p}]^{T}Q_{0}[R_{0}\tilde{v}-\tilde{y}] & = & c_{p}^{T}[I-Q_{0}Q_{0}^{T}]Q_{0}[R_{0}\tilde{v}-\tilde{y}]=0,\\
\mbox{and }[Q_{0}Q_{0}^{T}c_{p}-\tilde{y}]^{T}Q_{0}[R_{0}\tilde{v}-\tilde{y}] & = & [Q_{0}^{T}c_{p}-\tilde{y}]^{T}Q_{0}^{T}Q_{0}[R_{0}\tilde{v}-\tilde{y}]\\
 & = & [Q_{0}^{T}c_{p}-\tilde{y}]^{T}[R_{0}\tilde{v}-\tilde{y}]\leq0.
\end{eqnarray*}
The last inequality comes from the fact that $\tilde{y}=P_{\scriptsize\cone(-R_{0})}(Q_{0}^{T}c_{p})$.
Therefore ${[c_{p}-y]^{T}[v-y]\leq0}$ for all $v\in\cone(-N_{0})$,
which shows that $y=P_{\scriptsize\cone(-N_{0})}(c_{p})$ as needed.
The last statement is straightforward.
\end{proof}
The last result shows us that instead of following the description
of Algorithm \ref{alg:degen-GI-a-plus}, a dimension reduction can
reduce the size of the QP needed to be solved, which leads to a more
efficient algorithm.

\section{Conclusion}

We show how the inner GI step, which is derived from the dual active
set QP algorithm of \cite{Goldfarb_Idnani}, can allow us to keep
a QR factorization of the normals of the active constraints. The cost
of accumulating this QR factorization is actually quite cheap in view
of Remark \ref{rem:cost_GI_step}. Accumulating such information about
the normals of the active constraints accelerate the convergence to
a solution for the simple example in Section \ref{sec:simple-exa}.

Since keeping a large number of recently generated halfspaces is useful
for accelerating convergence, it is natural to look at how we can
keep the largest number of the most useful halfspaces efficiently.
This paper on integrating projection and QP algorithms appears to
be the first effort in this direction. We shall defer more serious
numerical calculations to a later paper.

\bibliographystyle{amsalpha}
\bibliography{refs}

\end{document}